\documentclass[a4paper,11pt,reqno]{amsart}
\usepackage{latexsym}
\usepackage{amssymb,amsfonts,amsmath,mathrsfs}
\newtheorem{theorem}{Theorem}[section]

\newtheorem{lemma}{Lemma}[section]

\newtheorem{remark}{Remark}[section]

\begin{document}

\title{Critical zeros of the Riemann zeta-function}
\author{H. M. Bui}
\address{School of Mathematics, University of Bristol, Bristol BS8 1TW, UK}
\email{hung.bui@bristol.ac.uk}

\begin{abstract}
In this unpublished note, we sketch an idea of using a three-piece mollifier to slightly improve the known percentages of zeros and simple zeros of the Riemann zeta-function on the critical line. This uses the recent result of Bettin \textit{et al.} [\textbf{\ref{BBLR}}] on the twisted fourth moment of the Riemann zeta-function.
\end{abstract}

\allowdisplaybreaks

\maketitle

\section{Introduction}
Let $\zeta(s)$ be the Riemann zeta-function. Let $N(T)$ denote the number of zeros of $\zeta(s)$, $\rho = \beta + i\gamma$, with $0 < \gamma \leq T$ counted with multiplicity. Also let $N_0(T)$ denote the number of such critical zeros with $\beta = 1/2$, and $N_0^*(T)$ denote the number of such critical  zeros with $\beta=1/2$ and being simple. Define $\kappa$ and $\kappa^*$ by
\begin{equation*}
 \kappa = \liminf_{T \rightarrow \infty} \frac{N_0(T)}{N(T)}, \qquad \kappa^* = \liminf_{T \rightarrow \infty} \frac{N_0^*(T)}{N(T)}.
\end{equation*}

Selberg [\textbf{\ref{S}}] was the first to prove that a positive proportion of zeros lie on the critical line. Following the approach of Levinson [\textbf{\ref{L}}] and the observation of Heath-Brown [\textbf{\ref{H-B}}], it is now known that [\textbf{\ref{BCY}},\textbf{\ref{F}}]
\begin{equation}\label{Fengs}
\kappa >0. 410725\qquad\textrm{and}\qquad \kappa^*>0.40582.
\end{equation}

\begin{remark}\emph{Bui \textit{et al.} [\textbf{\ref{BCY}}] showed that $\kappa>0.4105$ and $\kappa^*>0.40582$. Shortly after that, Feng [\textbf{\ref{F2}}; version 1] put a paper on arXiv claiming that $\kappa>0.4173$ and $\kappa^*>0.4075$. Feng also used Levinson's method, but instead chose the mollifier as a sum of various pieces of different shapes $\sum_{n\leq T^{\vartheta_1}}+\ldots+\sum_{n\leq T^{\vartheta_l}}$. At this point, Feng took $\vartheta_1=\ldots=\vartheta_l=4/7-\varepsilon$. The choice $\vartheta_1=4/7-\varepsilon$ was from Conrey [\textbf{\ref{C}}], but it was not clear why, say, $\vartheta_2<4/7$ is admissible. Conrey then emailed Feng requesting for the verification of that statement. Feng later agreed that that was a mistake, and subsequently replaced it with the second version. In this updated version [\textbf{\ref{F2}}; version 2], which is the same as the published one [\textbf{\ref{F}}], he chose $\vartheta_1=4/7-\varepsilon$ and $\vartheta_2=\ldots=\vartheta_l=1/2-\varepsilon$. With this Feng obtained $\kappa>0.4128$, but no claim on the lower bound for $\kappa^*$. Feng still did not give any explanation as to why, say, the range $\vartheta_2<1/2$ is admissible. This is doubtful and can be problematic. Note that if one just applies the result of Balasubramanian \textit{et al.} [\textbf{\ref{BCH-B}}] on the twisted second moment of the Riemann zeta-function to, say, the cross term $\int |\zeta(1/2+it)|^2\sum_{n\leq T^{\vartheta_1}}\overline{\sum_{n\leq T^{\vartheta_2}}}dt$, then one needs $\vartheta_1+\vartheta_2<1$. So without extra work, one can only take $\vartheta_1=4/7-\varepsilon$ and $\vartheta_2=\ldots=\vartheta_l=3/7-\varepsilon$ in Feng's paper. This numerically leads to the above bound $\kappa >0. 410725$ in \eqref{Fengs}.}
\end{remark}

In this paper we shall prove

\begin{theorem}
We have
\[
\kappa>0.410918\qquad and\qquad\kappa^*>0.40589.
\]
\end{theorem}

\begin{remark}\emph{Rigourously speaking this is not yet a theorem. Here we assume a result on the twisted third moment of the Riemann zeta-function (see Theorem 5.1 for the precise statement). We leave that unproved. It is possible that the ideas of Bettin \textit{et al.} [\textbf{\ref{BBLR}}] work in this context as well.}
\end{remark}

\section{Reduction to mean-value theorems}

\subsection{The mollifier}

To get lower bounds for $N_0(T)$ and $N_0^*(T)$ it suffices to consider a certain mollified second moment of the Riemann zeta-function and its derivatives.  This is well-known, so we shall simply state the conclusion.

Let $Q(x)$ be a real polynomial satisfying $Q(0) = 1$, and define
\begin{equation*}
 V(s) = Q\Big(-\frac{1}{\mathcal{L}} \frac{d}{ds} \Big) \zeta(s),
\end{equation*}
where for large $T$ we denote
\begin{equation*}
 \mathcal{L} = \log{T}.
\end{equation*}
Suppose $\psi(s)$ is a ``mollifier''.  Littlewood's lemma and the arithmetic-mean, geometric-mean inequality give
\begin{equation}\label{500}
 \kappa \geq 1 - \frac{1}{R} \log \bigg( \frac{1}{T} \int_{1}^{T} \big|V \psi(\sigma_0 + it)\big|^2 dt \bigg) + o(1),
\end{equation}
where $\sigma_0 =1/2 -R/\mathcal{L}$, and $R$ is a bounded positive real number to be chosen later.  Actually, by choosing $Q(x)$ to be a linear polynomial, we obtain a lower bound for the proportion of simple zeros, $\kappa^*$.

We choose a mollifier of the form
\begin{equation*}
\psi(s) =  \psi_1(s) + \psi_2(s) +\psi_3(s),
\end{equation*}
where $\psi_1$, $\psi_2$ and $\psi_3$ are mollifiers of quite different shapes.  Here $\psi_1$ is the usual mollifier
\begin{equation*}
 \psi_1(s) = \sum_{n \leq y_1} \frac{\mu(n) P_1\big(\frac{\log y_1/n}{\log y_1}\big)n^{\sigma_0-1/2}}{n^{s}}
\end{equation*}
and $\psi_2$ is the second piece mollifier used in [\textbf{\ref{BCY}}]
\[
\psi_2(s)=\chi(s+\tfrac{1}{2}-\sigma_0)\sum_{mn\leq y_2} \frac{\mu_2(m) P_2\big(\frac{\log y_2/mn}{\log y_2}\big)m^{\sigma_0-1/2}n^{1/2-\sigma_0}}{m^{s}n^{1-s}},
\]
where $\mu_2(m)$ are the coefficients of $1/\zeta^2(s)$, and $P_1(x)=\sum_{j}a_jx^j$ and $P_2(x)=\sum_{j}b_jx^j$ are certain polynomials satisfying $P_1(0)=0$ and $P_2(0)=P_2'(0)=P_2''(0)=0$. For the third mollifier, we take\footnote{It was originally an idea of K. Soundararajan to make use of the twisted fourth moment of the Riemann zeta-function and use a two-piece mollifier of the form $\psi_1+\psi_3$. The author learned about this from Brian Conrey, Chris Hughes and K. Soundararajan. However, the implied proportion of critical zeros using this two-piece mollifier has never been worked out explicitly.}
\begin{equation*}
 \psi_3(s) =\zeta(s+\tfrac{1}{2}-\sigma_0)\sum_{n \leq y_3} \frac{\mu_2(n)P_3\big(\frac{\log y_3/n}{\log y_3}\big)n^{\sigma_0-1/2}}{n^{s}},
\end{equation*}
where $P_3(x) = \sum_j c_j x^j$ is a certain polynomial satisfying $P_3(0) = P_3'(0)=P_3''(0)=P_3'''(0)= 0$. We also require the condition that $P_1(1)+P_3(1)=1$ (see the below remark). Throughout the paper we denote $y_1 = T^{\vartheta_1}$, $y_2 = T^{\vartheta_2}$ and $y_3 = T^{\vartheta_3}$, where $0<\vartheta_3<\vartheta_2 < \vartheta_1<1$ (we shall see later what conditions are required on $\vartheta_1$, $\vartheta_2$ and $\vartheta_3$). Note that formally
\[
\zeta(s)\sum_{n=1}^{\infty}\frac{\mu_2(n)}{n^s}=\frac{1}{\zeta(s)},
\]
which explains why $\psi_3(s)$ may also be a useful choice of a mollifier.

\begin{remark}
\emph{We can apply Levinson's method (see, for example, Appendix A of [\textbf{\ref{CIS}}]) to our choice of $\psi=\psi_1+\psi_2+\psi_3$. It is important to note that when applying Littlewood's lemma, we need to estimate the integral on the right side of a rectangle. Assume that $\psi_1(s)+\psi_3(s)$ can be expressed as a Dirichlet series, $\psi_1(s)+\psi_3(s)=\sum_{n}a(n)n^{-s}$, that integral is negligible given that $a(1)=1$. That is why we require $P_1(1)+P_3(1)=1$. }
\end{remark}

\subsection{A smoothing argument}

It simplifies some calculations to smooth out the integral in \eqref{500}. We introduce a smooth function $w(t)$ with the following properties
\begin{eqnarray*}
&& \textrm{i)}\ 0\leq w(t)\leq 1\ \textrm{for all}\ t\in\mathbb{R},\\
&& \textrm{ii)}\ w\ \textrm{has compact support in}\ [T/4,2T],\\
&& \textrm{iii)}\ w^{(j)}(t)\ll \Delta^{-j}\ \textrm{for each}\ j=0,1,2\ldots,\ \textrm{where}\ \Delta=T\mathcal{L}^{-1}.
\end{eqnarray*}

\begin{theorem}\label{main}
Suppose $\vartheta_1<4/7$, $\vartheta_2<1/2$ and $\vartheta_3 < 1/4$.  Then we have
\begin{equation}
\int_{-\infty}^{\infty} \big|V \psi(\sigma_0 + it)\big|^2w(t) dt = c(P_1,P_2,P_3,Q,R,\vartheta_1, \vartheta_2,\vartheta_3)\widehat{w}(0) +O_\varepsilon(T\mathcal{L}^{-1+\varepsilon}),
\end{equation}
where $c(P_1,P_2,P_3,Q,R,\vartheta_1, \vartheta_2,\vartheta_3) = c_1+c_2+c_3 + 2c_{12} + 2c_{23}+2c_{31}$, and the $c_{i}$ and $c_{ij}$ are given below by \eqref{c1}--\eqref{c31}.
\end{theorem}

\subsection{Numerical evaluations}

It is a standard exercise to deduce from Theorem \ref{main} the unsmoothed version
\begin{equation*}
\int_{1}^{T}\big|V \psi(\sigma_0 + it)\big|^2 dt = c(P_1,P_2,P_3,Q,R,\vartheta_1, \vartheta_2,\vartheta_3)T +O_\varepsilon(T\mathcal{L}^{-1+\varepsilon}).
\end{equation*}
Hence
\begin{equation*}
 \kappa \geq 1 - \frac{\log c(P_1,P_2,P_3,Q,R,\vartheta_1, \vartheta_2,\vartheta_3)}{R} + o(1).
\end{equation*}

Using Mathematica, with $R=1.26$,
\[
Q(x)=0.49068 + 0.61077(1 - 2x) -  0.14199(1 - 2x)^3 +0.04054(1 - 2x)^5,
\]
\[
P_1(x)=0.83651x + 0.09758x^2 - 0.29393x^3 + 0.73372x^4 - 0.3753x^5,
\]
\[
P_2(x)=0.0237x^3 - 0.00744x^4 + 0.00174x^5
\]
and
\[
P_3(x)=0.00155x^4-0.00013x^5,
\]
we get $\kappa>0.410918$. To get $\kappa^*>0.40589$ we take $R=1.12$, $Q(x)=1-1.03232x$,
\[
P_1(x)=0.82653x + 0.02626x^2 - 0.00774x^3 + 0.34803x^4 - 0.19371x^5,
\]
\[
P_2(x)=0.0324x^3 - 0.00759x^4 + 0.00742x^5
\]
and
\[
P_3(x)=0.00094x^4-0.00031x^5.
\]

\section{The mean-value results}

Writing $\psi = \psi_1 + \psi_2+\psi_3$ and opening the square, we get
\begin{eqnarray*}
 \int \big|V \psi\big|^2w& =& \int \big| V \psi_1\big|^2w+ \int \big|V \psi_2\big|^2w+ \int \big|V \psi_3\big|^2w\\
&&\qquad + 2\textrm{Re}\bigg\{\int \big|V\big|^2 \psi_1\overline{ \psi_2}w\bigg\} + 2\textrm{Re}\bigg\{\int \big|V\big|^2 \psi_2\overline{ \psi_3}w\bigg\}+2\textrm{Re}\bigg\{\int \big|V\big|^2 \psi_3\overline{ \psi_1}w\bigg\} \\
&=:& I_1+I_2+I_3+ 2\textrm{Re}(I_{12}) +2\textrm{Re}(I_{23})+2\textrm{Re}(I_{31})  .
\end{eqnarray*}
We shall compute the integrals in turn.  It turns out that $I_{12}$, $I_{23}$ and $I_{31}$ are asymptotically real.

\subsection{The main terms}

First we quote Theorem 2 of Conrey [\textbf{\ref{C}}]. Conrey's theorem is for the unsmoothed version, but the following smoothed version follows easily from that.

\begin{theorem}[Conrey]  Suppose $\vartheta_1 < 4/7$ and $P_1(0)=0$. Then we have
\begin{equation*}
\int_{-\infty}^{\infty} \big |V\psi_1(\sigma_0 + it)\big|^2w(t) dt =c_{1}\widehat{w}(0)+O_\varepsilon(T\mathcal{L}^{-1+\varepsilon}),
\end{equation*}
where
\begin{equation}\label{c1}
 c_{1} = P_1(1)^2 + \tfrac{1}{\vartheta_1} \int_0^{1} \int_0^{1} e^{2Rt}\Big (Q(t) P_1'(u) + \vartheta_1 Q'(t) P_1(u) + \vartheta_1 R Q(t) P_1(u)\Big)^2 dtdu.
\end{equation}
\end{theorem}

The terms $I_{12}$ and $I_2$ are evaluated in [\textbf{\ref{BCY}}].

\begin{theorem}[Bui, Conrey and Young]
Suppose $\vartheta_2 < \vartheta_1  < 4/7$ and $P_1(0)=0=P_2(0)=P_2'(0)$. Then we have
\begin{equation*}
\int_{-\infty}^{\infty}  |V|^2 \psi_1\overline{\psi_2}   (\sigma_0 + it)w(t) dt = c_{12} \widehat{w}(0) + O_\varepsilon(T\mathcal{L}^{-1+\varepsilon}),
\end{equation*}
where
\begin{eqnarray}
\label{c12}
&&\!\!\!\!\!\!\!\!\!\!\!\!c_{12}=  \tfrac{4 \vartheta_2^2}{\vartheta_1^2}
 \frac{d^2}{dx_1 dx_2}\bigg[ \int_0^{1} \mathop{\int \int}_{\substack{t_1+t_2 \leq u \\t_1,t_2\geq0}} 
   e^{R\big(1-\vartheta_1(x_1-x_2) + \vartheta_2(t_1-t_2)\big)}(1-u)  Q(-\vartheta_1x_1  + \vartheta_2t_1 )
\nonumber\\
 &&\!\!\!\!\!\!\!\!\!\!\!\!\qquad Q(1+\vartheta_1x_2  -\vartheta_2t_2 ) P_1\Big(x_1+x_2+ 1-\tfrac{ \vartheta_2}{\vartheta_1}(1-u)\Big) P_2''(u-t_1-t_2) dt_1dt_2du \bigg]_{x_1=x_2=0} .
\end{eqnarray}
\end{theorem}

\begin{theorem}[Bui, Conrey and Young]
\label{thm:mainterm2}
Suppose $\vartheta_2 < \tfrac{1}{2}$ and $P_2(0)=P_2'(0)=P_2''(0)=0$. Then we have
\begin{equation*}
\int_{-\infty}^{\infty}  \big |V\psi_2(\sigma_0 + it)\big|^2w(t) dt = c_2 \widehat{w}(0) + O_\varepsilon(T\mathcal{L}^{-1+\varepsilon}),
\end{equation*}
where
\begin{eqnarray}\label{c2}
&&c_2 = \frac{2}{3\vartheta_2} \frac{d^4}{dx_1^2 dx_2^2}
\bigg[\mathop{\int}_{[0,1]^4}e^{-\vartheta_2 R\big(x_1+x_2-t_1(x_1+u)-t_2(x_2+u)\big)+2Rt_3\big(1+\vartheta_2(x_1+x_2-t_1(x_1+u)-t_2(x_2+u))\big)}\nonumber\\
&&\qquad \Big(1+\vartheta_2\big(x_1+x_2-t_1(x_1+u)-t_2(x_2+u)\big)\Big)(x_1+u)(x_2+u)(1-u)^{4}\nonumber  \\
&&\qquad\qquad Q\bigg(\vartheta_2\big(-x_1 + t_2(x_2+u)\big) + t_3\Big(1+\vartheta_2\big(x_1+x_2-t_1(x_1+u)-t_2(x_2+u)\big)\Big)\bigg)\nonumber \\
&&\qquad\qquad\qquad Q\bigg(\vartheta_2\big(-x_2 + t_1(x_1+u)\big) + t_3\Big(1+\vartheta_2\big(x_1+x_2-t_1(x_1+u)-t_2(x_2+u)\big)\Big)\bigg)\nonumber\\
&&\qquad \qquad\qquad\qquad P_2''\big((x_1 + u)(1-t_1))\big) P_2''\big((x_2+u)(1-t_2)\big) dt_1 dt_2 dt_3 du\bigg]_{x_1=x_2=0}.
\end{eqnarray}
\end{theorem}

We are left to evaluate $I_3$, $I_{23}$ and $I_{31}$, which are
\begin{eqnarray*}
&&I_3=\int_{-\infty}^{\infty}  \big|V (\sigma_0 + it)\zeta(\tfrac{1}{2}+it)\big|^2\sum_{m\leq y_3} \frac{\mu_2(m)P_3\big(\frac{\log y_3/m}{\log y_3}\big)}{m^{1/2+it}}\sum_{l\leq y_3} \frac{\mu_2(l)P_3\big(\frac{\log y_3/l}{\log y_3}\big)}{l^{1/2-it}}w(t) dt,\\
&&I_{23}=\int_{-\infty}^{\infty}  \big|V (\sigma_0 + it)\big|^2\zeta(\tfrac{1}{2}+it)\sum_{mn\leq y_2} \frac{\mu_2(m)P_2\big(\frac{\log y_2/mn}{\log y_2}\big)}{m^{1/2+it}n^{1/2-it}}\sum_{l\leq y_3} \frac{\mu_2(l)P_3\big(\frac{\log y_3/l}{\log y_3}\big)}{l^{1/2-it}}w(t) dt
\end{eqnarray*}
and
\[
I_{31}=\int_{-\infty}^{\infty}  \big|V (\sigma_0 + it)\big|^2\zeta(\tfrac{1}{2}+it)\sum_{m\leq y_1} \frac{\mu(m)P_1\big(\frac{\log y_1/m}{\log y_1}\big)}{m^{1/2-it}}\sum_{l\leq y_3} \frac{\mu_2(l)P_3\big(\frac{\log y_3/l}{\log y_3}\big)}{l^{1/2+it}}w(t) dt.
\]

\begin{theorem}\label{thc3}
Suppose $\vartheta_3<1/4$ and $P_3(0)=P_3'(0)=P_3''(0)=P_3'''(0)=0$. Then we have
\begin{equation*}
I_{3} =c_{3}\widehat{w}(0) + O_\varepsilon(T\mathcal{L}^{-1+\varepsilon}),
\end{equation*}
where
\begin{eqnarray}\label{c3}
&&\!\!\!\!\!\!\!\!\!\!c_{3}=\tfrac{1}{12\vartheta_3^4}\frac{d^8}{dx_1^2dx_2^2dx_3^2dx_4^2}\bigg[\mathop{\int}_{[0,1]^5}e^{-\vartheta_3R(x_2+x_3)+Rt_1(1-t_4)(1+\vartheta_3(x_1+x_3))+Rt_2(1-t_3)(1+\vartheta_3(x_2+x_4))}\nonumber\\
&&\!\!\!\!\!\!\!\!\!\!\quad e^{Rt_3\big(t_1+\vartheta_3(-x_1+x_2+(x_1+x_3)t_1)\big)+Rt_4\big(t_2+\vartheta_3(x_3-x_4+(x_2+x_4)t_2)\big)}\nonumber\\
&&\!\!\!\!\!\!\!\!\!\!\quad\quad\Big(t_1-t_2+\vartheta_3\big(-x_1+x_2+(x_1+x_3)t_1-(x_2+x_4)t_2\big)\Big)\nonumber\\
&&\!\!\!\!\!\!\!\!\!\!\quad\quad\quad\Big(t_1-t_2+\vartheta_3\big(-x_3+x_4+(x_1+x_3)t_1-(x_2+x_4)t_2\big)\Big)\nonumber\\
&&\!\!\!\!\!\!\!\!\!\!\quad\quad\quad\quad\big(1+\vartheta_3(x_1+x_3)\big)\big(1+\vartheta_3(x_2+x_4)\big)(1-u)^3\nonumber\\
&&\!\!\!\!\!\!\!\!\!\!\quad\quad\quad\quad\quad Q\bigg(-\vartheta_3 x_2+t_2(1-t_3)\big(1+\vartheta_3(x_2+x_4)\big)+t_3\Big(t_1+\vartheta_3\big(-x_1+x_2+(x_1+x_3)t_1\big)\Big)\bigg)\nonumber\\
&&\!\!\!\!\!\!\!\!\!\!\quad\quad\quad\quad\quad\quad Q\bigg(-\vartheta_3 x_3+t_1(1-t_4)\big(1+\vartheta_3(x_1+x_3)\big)+t_4\Big(t_2+\vartheta_3\big(x_3-x_4+(x_2+x_4)t_2\big)\Big)\bigg)\nonumber\\
&&\!\!\!\!\!\!\!\!\!\!\quad\quad\quad\quad\quad\quad\quad P_3(x_1+x_2+u)P_3(x_3+x_4+u)dt_1dt_2dt_3dt_4du\bigg]_{\underline{x}=0}.
\end{eqnarray}
\end{theorem}

\begin{theorem}\label{thc23}
Suppose $\vartheta_2<1/2$, $\vartheta_3<1/4$ and $P_2(0)=P_2'(0)=P_2''(0)=0=P_3(0)=P_3'(0)=P_3''(0)=P_3'''(0)$. Then we have
\begin{equation*}
I_{23} =c_{23}\widehat{w}(0) + O_\varepsilon(T\mathcal{L}^{-1+\varepsilon}),
\end{equation*}
where
\begin{eqnarray}\label{c23}
&&\!\!\!\!\!\!\!\!\!\!c_{23}=\tfrac{2}{3\vartheta_{2}^2}\frac{d^6}{dx_1^2dx_2^2dx_3^2}\bigg[\mathop{\int}_{[0,1]^5}e^{-R(\vartheta_2 x_1+\vartheta_3 x_2)+Rt_1t_2(1-t_3-t_3t_4)(\vartheta_2(1+x_1)-\vartheta_3(1-u))}\nonumber\\
&&\!\!\!\!\!\!\!\!\!\!\ \quad e^{Rt_3(1+t_4)(1+\vartheta_2 x_1+\vartheta_3 x_2)+R(1-t_4)\big(\vartheta_3 (x_2-x_3)-t_1(2t_2-1)(\vartheta_2(1+x_1)-\vartheta_3(1-u))\big)}\nonumber\\
&&\!\!\!\!\!\!\!\!\!\!\ \ \quad\quad\bigg(-\vartheta_3\big(x_2-x_3)+t_1(2t_2-1)\big(\vartheta_2(1+x_1)-\vartheta_3(1-u)\big)\nonumber\\
&&\qquad\qquad\qquad\qquad\qquad+t_3\Big(1+\vartheta_2 x_1+\vartheta_3 x_2-t_1t_2\big(\vartheta_2(1+x_1)-\vartheta_3(1-u)\big)\Big)\bigg)\nonumber\\
&&\!\!\!\!\!\!\!\!\!\!\ \ \ \quad\quad\quad\Big(1+\vartheta_2 x_1+\vartheta_3 x_2-t_1t_2\big(\vartheta_2(1+x_1)-\vartheta_3(1-u)\big)\Big)t_1\Big(x_1+1-\tfrac{\vartheta_3}{\vartheta_2}(1-u)\Big)^2(1-u)^3\nonumber\\
&&\!\!\!\!\!\!\!\!\!\!\quad\ \ \ \ \quad\quad\quad Q\bigg(-\vartheta_3 x_2+t_1t_2(1-t_3t_4)\big(\vartheta_2(1+x_1)-\vartheta_3(1-u)\big)+t_3t_4(1+\vartheta_2 x_1+\vartheta_3 x_2)\nonumber\\
&&\qquad\qquad\qquad\qquad\qquad+(1-t_4)\Big(\vartheta_3 (x_2-x_3)-t_1(2t_2-1)\big(\vartheta_2(1+x_1)-\vartheta_3(1-u)\big)\Big)\bigg)\nonumber\\
&&\!\!\!\!\!\!\!\!\!\!\quad\quad\ \ \ \ \quad\quad\quad Q\bigg(-\vartheta_2x_1+t_3\Big(1+\vartheta_2 x_1+\vartheta_3 x_2-t_1t_2\big(\vartheta_2(1+x_1)-\vartheta_3(1-u)\big)\Big)\bigg)\nonumber\\
&&\!\!\!\!\!\!\!\!\!\!\quad\quad\quad\ \ \ \ \ \quad\quad\quad P_2''\Big(\big(x_1+1-\tfrac{\vartheta_3}{\vartheta_2}(1-u)\big)(1-t_1)\Big)P_3\big(x_2+x_3+u\big)dt_1dt_2dt_3dt_4du\bigg]_{\underline{x}=0}.
\end{eqnarray}
\end{theorem}

\begin{theorem}\label{thc31}
Suppose $\vartheta_1<4/7$, $\vartheta_3<1/4$ and $P_1(0)=0=P_3(0)=P_3'(0)$. Then we have
\begin{equation*}
I_{31} = c_{31}\widehat{w}(0) + O_\varepsilon(T\mathcal{L}^{-1+\varepsilon}),
\end{equation*}
where
\begin{eqnarray}\label{c31}
&&\!\!\!\!\!\!\!\!\!\!\!\!c_{31}=\tfrac{1}{\vartheta_{1}^{2}}\frac{d^4}{dx_1dx_2dx_3^2}\bigg[\mathop{\int}_{[0,1]^3}e^{-R(\vartheta_1 x_2+\vartheta_3 x_3)+Rt_1(1+t_2)(1+\vartheta_1 x_1+\vartheta_3 x_3)-\vartheta_1Rt_2(x_1-x_2)}\nonumber\\
&&\!\!\!\!\!\!\!\!\!\!\!\!\quad \big(-\vartheta_1(x_1-x_2)+t_1(1+\vartheta_1 x_1+\vartheta_3 x_3)\Big)(1+\vartheta_1 x_1+\vartheta_3 x_3)(1-u)\nonumber\\
&&\!\!\!\!\!\!\!\!\!\!\!\!\quad\quad Q\Big(-\vartheta_1 x_2+t_1t_2(1+\vartheta_1 x_1+\vartheta_3 x_3)-\vartheta_1t_2(x_1-x_2)\Big)Q\Big(-\vartheta_3 x_3+t_1(1+\vartheta_1 x_1+\vartheta_3 x_3)\Big)\nonumber\\
&&\!\!\!\!\!\!\!\!\!\!\!\!\quad\quad\quad\ P_1\Big(x_1+x_2+1-\tfrac{\vartheta_3}{\vartheta_1}(1-u)\Big)P_3(x_3+u)dt_1dt_2du\bigg]_{\underline{x}=0}.
\end{eqnarray}
\end{theorem}

\subsection{The shift parameters}

Rather than working directly with $V(s)$, we shall instead consider the following three general integrals
\begin{eqnarray}\label{I3}
I_3(\alpha,\beta,\gamma,\delta)&=&\int_{-\infty}^{\infty}  \zeta(\tfrac{1}{2}+\alpha+it) \zeta(\tfrac{1}{2}+\beta+it) \zeta(\tfrac{1}{2}+\gamma-it)\zeta(\tfrac{1}{2}+\delta-it)\nonumber\\
&&\qquad\qquad \sum_{m\leq y_3} \frac{\mu_2(m)P_3\big(\frac{\log y_3/m}{\log y_3}\big)}{m^{1/2+it}}\sum_{l\leq y_3} \frac{\mu_2(l)P_3\big(\frac{\log y_3/l}{\log y_3}\big)}{l^{1/2-it}}w(t) dt,
\end{eqnarray}
\begin{eqnarray}\label{I23}
I_{23}(\alpha,\beta,\gamma)&=&\int_{-\infty}^{\infty}  \zeta(\tfrac{1}{2}+\alpha+it) \zeta(\tfrac{1}{2}+\beta+it) \zeta(\tfrac{1}{2}+\gamma-it)\nonumber\\
&&\qquad\qquad\sum_{mn\leq y_2} \frac{\mu_2(m)P_2\big(\frac{\log y_2/mn}{\log y_2}\big)}{m^{1/2+it}n^{1/2-it}}\sum_{l\leq y_3} \frac{\mu_2(l)P_3\big(\frac{\log y_3/l}{\log y_3}\big)}{l^{1/2-it}}w(t) dt
\end{eqnarray}
and
\begin{eqnarray}\label{I31}
I_{31}(\alpha,\beta,\gamma)&=&\int_{-\infty}^{\infty}  \zeta(\tfrac{1}{2}+\alpha+it) \zeta(\tfrac{1}{2}+\beta+it) \zeta(\tfrac{1}{2}+\gamma-it)\nonumber\\
&&\qquad\qquad\sum_{m\leq y_1} \frac{\mu(m)P_1\big(\frac{\log y_1/m}{\log y_1}\big)}{m^{1/2-it}}\sum_{l\leq y_3} \frac{\mu_2(l)P_3\big(\frac{\log y_3/l}{\log y_3}\big)}{l^{1/2+it}}w(t) dt.
\end{eqnarray}

Our main goal in the rest of the paper is to prove the following lemmas.

\begin{lemma}\label{lemmac3}
Suppose $\vartheta_3<1/4$ and $P_3(0)=P_3'(0)=P_3''(0)=P_3'''(0)=0$. Then we have
\begin{equation*}
I_{3}(\alpha,\beta,\gamma,\delta) =c_{3}(\alpha,\beta,\gamma,\delta)\widehat{w}(0) + O_\varepsilon(T\mathcal{L}^{-1+\varepsilon}),
\end{equation*}
uniformly for $\alpha, \beta,\gamma,\delta \ll \mathcal{L}^{-1}$, where $c_{3}(\alpha,\beta,\gamma,\delta)$ is given by
\begin{eqnarray}\label{c3a}
&&\tfrac{1}{12\vartheta_3^4}\frac{d^8}{dx_1^2dx_2^2dx_3^2dx_4^2}\bigg[\mathop{\int}_{[0,1]^5}y_3^{\alpha x_1+\beta x_2+\gamma x_3+\delta x_4}(Ty_3^{x_1+x_3})^{-(\alpha+\gamma)t_1}(Ty_3^{x_2+x_4})^{-(\beta+\delta)t_2}\nonumber\\
&&\quad\Big(T^{t_1-t_2}y_3^{-x_1+x_2+(x_1+x_3)t_1-(x_2+x_4)t_2}\Big)^{-(\beta-\alpha)t_3}\Big(T^{t_1-t_2}y_3^{-x_3+x_4+(x_1+x_3)t_1-(x_2+x_4)t_2}\Big)^{-(\delta-\gamma)t_4}\nonumber\\
&&\quad\quad\Big(t_1-t_2+\vartheta_3\big(-x_1+x_2+(x_1+x_3)t_1-(x_2+x_4)t_2\big)\Big)\\
&&\quad\quad\quad\Big(t_1-t_2+\vartheta_3\big(-x_3+x_4+(x_1+x_3)t_1-(x_2+x_4)t_2\big)\Big)\big(1+\vartheta_3(x_1+x_3)\big)\nonumber\\
&&\quad\quad\quad\quad\big(1+\vartheta_3(x_2+x_4)\big)(1-u)^3P_3(x_1+x_2+u)P_3(x_3+x_4+u)dt_1dt_2dt_3dt_4du\bigg]_{\underline{x}=0}.\nonumber
\end{eqnarray}
\end{lemma}

\begin{remark}
\emph{In the special case $\alpha=\beta=\gamma=\delta=0$, Lemma \ref{lemmac3} agrees with the mollified fourth moment of the Riemann zeta-function predicted by Conrey and Snaith using the ratios conjecture [\textbf{\ref{CS}}; Theorem 6.1].}
\end{remark}

\begin{lemma}\label{lemmac23}
Suppose $\vartheta_2<1/2$, $\vartheta_3<1/4$ and $P_2(0)=P_2'(0)=P_2''(0)=0=P_3(0)=P_3'(0)=P_3''(0)=P_3'''(0)$. Then we have
\begin{equation*}
I_{23}(\alpha,\beta,\gamma) =c_{23}(\alpha,\beta,\gamma)\widehat{w}(0) + O_\varepsilon(T\mathcal{L}^{-1+\varepsilon}),
\end{equation*}
uniformly for $\alpha, \beta,\gamma \ll \mathcal{L}^{-1}$, where $c_{23}(\alpha,\beta,\gamma)$ is given by
\begin{eqnarray}\label{c23a}
&&\tfrac{2}{3\vartheta_{2}^2}\frac{d^6}{dx_1^2dx_2^2dx_3^2}\bigg[\mathop{\int}_{[0,1]^5}y_{2}^{\gamma x_1-\big(\alpha t_2+\beta  (1-t_2)\big)t_1(1+x_1)}y_{3}^{\alpha x_2+\beta x_3+\big(\alpha t_2+\beta  (1-t_2)\big)t_1(1-u)}\nonumber\\
&&\quad\ \big(Ty_{2}^{x_1-t_1t_2(1+x_1)}y_{3}^{x_2+t_1t_2(1-u)}\big)^{-(\alpha+\gamma)t_3}\nonumber\\
&&\quad\ \ \quad\Big(y_{2}^{t_1(2t_2-1)(1+x_1)}y_{3}^{-x_2+ x_3-t_1(2 t_2-1)(1-u)}\big(Ty_{2}^{x_1-t_1t_2(1+x_1)}y_{3}^{x_2+t_1t_2(1-u)}\big)^{t_3}\Big)^{-(\beta-\alpha)t_4}\nonumber\\
&&\quad\ \ \ \quad\quad\bigg(-\vartheta_3\big(x_2-x_3)+t_1(2t_2-1)\big(\vartheta_2(1+x_1)-\vartheta_3(1-u)\big)\\
&&\qquad\qquad\qquad\qquad\qquad\qquad+t_3\Big(1+\vartheta_2 x_1+\vartheta_3 x_2-t_1t_2\big(\vartheta_2(1+x_1)-\vartheta_3(1-u)\big)\Big)\bigg)\nonumber\\
&&\quad\quad\ \ \ \ \quad\quad\Big(1+\vartheta_2 x_1+\vartheta_3 x_2-t_1t_2\big(\vartheta_2(1+x_1)-\vartheta_3(1-u)\big)\Big)t_1\Big(x_1+1-\tfrac{\vartheta_3}{\vartheta_2}(1-u)\Big)^2\nonumber\\
&&\quad\quad\quad\ \ \ \ \ \quad\quad (1-u)^3P_2''\Big(\big(x_1+1-\tfrac{\vartheta_3}{\vartheta_2}(1-u)\big)(1-t_1)\Big)P_3\big(x_2+x_3+u\big)dt_1dt_2dt_3du\bigg]_{\underline{x}=0}.\nonumber
\end{eqnarray}
\end{lemma}

\begin{lemma}\label{lemmac31}
Suppose $\vartheta_1<4/7$, $\vartheta_3<1/4$ and $P_1(0)=0=P_3(0)=P_3'(0)$. Then we have
\begin{equation*}
I_{31}(\alpha,\beta,\gamma) = c_{31}(\alpha,\beta,\gamma) \widehat{w}(0) + O_\varepsilon(T\mathcal{L}^{-1+\varepsilon}),
\end{equation*}
uniformly for $\alpha, \beta,\gamma \ll \mathcal{L}^{-1}$, where $c_{31}(\alpha,\beta,\gamma)$ is given by
\begin{eqnarray}\label{c31a}
&&\tfrac{1}{\vartheta_{1}^{2}}\frac{d^4}{dx_1dx_2dx_3^2}\bigg[\mathop{\int}_{[0,1]^3}y_{1}^{\alpha x_1+\beta x_2}y_{3}^{\gamma x_3}(Ty_{1}^{x_1}y_{3}^{x_3})^{-(\alpha+\gamma)t_1}\big(y_{1}^{-x_1+x_2}(Ty_{1}^{x_1}y_{3}^{x_3})^{t_1}\big)^{-(\beta-\alpha)t_2}\nonumber\\
&&\qquad\qquad \Big(-\vartheta_1(x_1-x_2)+t_1(1+\vartheta_1 x_1+\vartheta_3 x_3)\Big)(1+\vartheta_1 x_1+\vartheta_3 x_3)\\
&&\qquad\qquad\qquad\qquad (1-u)P_1\Big(x_1+x_2+1-\tfrac{\vartheta_3}{\vartheta_1}(1-u)\Big)P_3(x_3+u)dt_1dt_2du\bigg]_{\underline{x}=0}.\nonumber
\end{eqnarray}
\end{lemma}

We now prove that Theorems \ref{thc3}--\ref{thc31} follow from Lemmas \ref{lemmac3}--\ref{lemmac31}, respectively.
Let $I_{\star}$ denote either $I_{23}$ or $I_{31}$. We first note that
\begin{equation}\label{eq:diffop1}
I_{3} =   Q\Big(\frac{-1}{\mathcal{L}} \frac{d}{d\beta}\Big) Q\Big(\frac{-1}{\mathcal{L}} \frac{d}{d\gamma}\Big) I_{3}(\alpha, \beta,\gamma,\delta) \bigg|_{\alpha=\delta=0,\ \beta=\gamma=-R/\mathcal{L}}
\end{equation}
and
\begin{equation}
\label{eq:diffop}
I_{\star} =   Q\Big(\frac{-1}{\mathcal{L}} \frac{d}{d\beta}\Big) Q\Big(\frac{-1}{\mathcal{L}} \frac{d}{d\gamma}\Big) I_{\star}(\alpha, \beta,\gamma) \bigg|_{\alpha=0,\ \beta=\gamma=-R/\mathcal{L}}.
\end{equation}
We then argue that we can obtain either $c_3$ or $c_{\star}$ by applying the above differential operator to the corresponding $c_{3}(\alpha, \beta,\gamma,\delta)$ or $c_{\star}(\alpha, \beta,\gamma)$.  Since $I_{3}(\alpha, \beta,\gamma,\delta)$, $I_{\star}(\alpha, \beta,\gamma)$, $c_{3}(\alpha, \beta,\gamma,\delta)$ and $c_{\star}(\alpha, \beta,\gamma)$ are holomorphic with respect to $\alpha, \beta,\gamma,\delta$ small, the derivatives appearing in \eqref{eq:diffop1} and \eqref{eq:diffop} can be obtained as integrals of radii $\asymp \mathcal{L}^{-1}$ around the points $\alpha=\delta=0$, $\beta=\gamma=-R/\mathcal{L}$, using Cauchy's integral formula.  Since the error terms hold uniformly on these contours, the same error terms that hold for $I_{3}(\alpha, \beta,\gamma,\delta)$ and $I_{\star}(\alpha, \beta,\gamma)$ also hold for $I_3$ and $I_{\star}$.

Next we check that applying the above differential operator to $c_{3}(\alpha, \beta,\gamma,\delta)$ and $c_{\star}(\alpha, \beta,\gamma)$ does indeed give $c_3$ and $c_{\star}$.
Notice the formula
\begin{equation}
\label{eq:Qop}
 Q\Big(\frac{-1}{\mathcal{L}} \frac{d}{d\beta}\Big)X^{-\beta}=Q\Big(\frac{\log X}{\mathcal{L}}\Big)X^{-\beta}.
\end{equation}
Using \eqref{eq:Qop} and \eqref{c31a}, we have
\begin{eqnarray*}
 &&\!\!\!\!\!\!\!\!\!\!\!\!Q\Big(\frac{-1}{\mathcal{L}} \frac{d}{d\beta}\Big) Q\Big(\frac{-1}{\mathcal{L}} \frac{d}{d\gamma}\Big) c_{31}(0,\beta,\gamma) =\tfrac{1}{\vartheta_{1}^{2}}\frac{d^4}{dx_1dx_2dx_3^2}\bigg[\mathop{\int}_{[0,1]^3}y_{1}^{\beta x_2}y_{3}^{\gamma x_3}(Ty_{1}^{x_1}y_{3}^{x_3})^{-\gamma t_1}\nonumber\\
&&\!\!\!\!\!\!\!\!\!\!\!\!\quad \big(y_{1}^{-x_1+x_2}(Ty_{1}^{x_1}y_{3}^{x_3})^{t_1}\big)^{-\beta t_2}\Big(-\vartheta_1(x_1-x_2)+t_1(1+\vartheta_1 x_1+\vartheta_3 x_3)\Big)(1+\vartheta_1 x_1+\vartheta_3 x_3)\\
&&\!\!\!\!\!\!\!\!\!\!\!\!\quad\quad(1-u) Q\Big(-\vartheta_1 x_2-\vartheta_1t_2(x_1-x_2)+t_1t_2(1+\vartheta_1 x_1+\vartheta_3 x_3)\Big)\\
&&\!\!\!\!\!\!\!\!\!\!\!\!\quad\quad\quad Q\Big(-\vartheta_3 x_3+t_1(1+\vartheta_1 x_1+\vartheta_3 x_3)\Big)P_1\Big(x_1+x_2+1-\tfrac{\vartheta_3}{\vartheta_1}(1-u)\Big)P_3(x_3+u)dt_1dt_2du\bigg]_{\underline{x}=0}.
\end{eqnarray*}
Setting $\beta = \gamma = -R/\mathcal{L}$ and simplifying 
gives \eqref{c31}.  A similar argument produces \eqref{c23} from \eqref{c23a} and produces \eqref{c3} from \eqref{c3a}.

We prove Lemma \ref{lemmac31} in Section \ref{section:c31}, Lemma \ref{lemmac23} in Section \ref{section:c23} and Lemma \ref{lemmac3} in Section \ref{section:c3}.

\section{Various lemmas}

\subsection{The Euler-Maclaurin formula}

The following two lemmas are easy consequences of the Euler-Maclaurin formula, see [\textbf{\ref{BCY}}; Lemmas 4.4 and 4.6].

\begin{lemma}\label{600}
Suppose $y_2\leq y_1$, $|z|\ll(\log y_2)^{-1}$, and that $f_1$ and $f_2$ are smooth functions. Then we have
\begin{eqnarray*}
&&\sum_{n\leq y_2}\frac{d_k(n)}{n}\Big(\frac{y_2}{n}\Big)^{z}f_1\Big(\frac{\log y_1/n}{\log y_1}\Big)f_2\Big(\frac{\log y_2/n}{\log y_2}\Big)\\
&&\qquad=\frac{(\log y_2)^k}{(k-1)!}\int_{0}^{1}y_{2}^{zu}(1-u)^{k-1}f_1\Big(1-\frac{(1-u)\log y_2}{\log y_1}\Big)f_2(u)du+O\big((\log y_2)^{k-1}\big).
\end{eqnarray*}
\end{lemma}

\begin{lemma}\label{601}
Suppose $-1\leq\sigma\leq0$. Then we have
\begin{equation*}
\sum_{n\leq y_1}\frac{d_k(n)}{n}\Big(\frac{y_1}{n}\Big)^\sigma\ll (\log y_1)^{k-1}\min\big\{|\sigma|^{-1},\log y_1\big\}.
\end{equation*}
\end{lemma}

\subsection{Mellin pairs}

By convention, we set $P_j(x) = 0$ for $j=1,2,3$ and $x\leq0$. Note that with this definition we have
\begin{equation}\label{Mellin}
P_1\Big(\frac{\log y_1/n}{\log y_1}\Big)=\sum_{j}\frac{a_j j!}{(\log y_1)^j}\frac{1}{2\pi i}\int_{(1)}\Big(\frac{y_1}{n}\Big)^u\frac{du}{u^{j+1}}
\end{equation}
for all $n$. Similar expressions holds for $P_2\big(\frac{\log y_2/n}{\log y_2}\big)$ and $P_3\big(\frac{\log y_3/n}{\log y_3}\big)$.

\section{Proof of Lemma \ref{lemmac31}}\label{section:c31}

\subsection{Reduction to a contour integral}

We shall used the following unproved twisted third moment of the Riemann zeta-function. 

\begin{theorem}
Suppose $H$ and $K$ satisfy $H\leq T^{4/7-\varepsilon}$ and $K\leq T^{1/4-\varepsilon}$, or $HK\leq T^{3/4-\varepsilon}$ and $K\leq T^{1/2-\varepsilon}$\footnote{Theorem 5.1 is likely to hold for larger ranges of $H$ and $K$. These conditions, however, suffice for our purposes.}. Then we have
\begin{eqnarray*}
&&\sum_{\substack{h\leq H\\k\leq K}}\frac{a_h\overline{a_k}}{\sqrt{hk}}\int_{-\infty}^{\infty}\zeta(\tfrac{1}{2}+\alpha+it)\zeta(\tfrac{1}{2}+\beta+it)\zeta(\tfrac{1}{2}+\gamma-it)\Big(\frac{h}{k}\Big)^{it}w(t)dt\\
&&\qquad\qquad\qquad=\sum_{\substack{h\leq H\\k\leq K}}\frac{a_h\overline{a_k}}{\sqrt{hk}}\int_{-\infty}^{\infty}w(t)\bigg\{Z_{\alpha,\beta,\gamma}(h,k)+\Big(\frac{t}{2\pi}\Big)^{-(\alpha+\gamma)}Z_{-\gamma,\beta,-\alpha}(h,k)\\
&&\qquad\qquad\qquad\qquad\qquad\qquad\qquad\qquad\qquad+\Big(\frac{t}{2\pi}\Big)^{-(\beta+\gamma)}Z_{\alpha,-\gamma,-\beta}(h,k)\bigg\}dt+O_\varepsilon(T^{1-\varepsilon})
\end{eqnarray*}
uniformly for $\alpha,\beta,\gamma\ll \mathcal{\mathcal{L}}^{-1}$, where 
\[
Z_{\alpha,\beta,\gamma}(h,k)=\sum_{kab=hc}\frac{1}{a^{1/2+\alpha}b^{1/2+\beta}c^{1/2+\gamma}}.
\]
\end{theorem}

Recall that $I_{31}(\alpha, \beta,\gamma)$ is defined by \eqref{I31}.  We have
\[
I_{31}(\alpha, \beta,\gamma)=I_{31}^1+I_{31}^2+I_{31}^3+O_\varepsilon(T^{1-\varepsilon}),
\]
where
\begin{equation*}
 I_{31}^1= \widehat{w}(0)
\sum_{l,m} 
\frac{\mu(m) \mu_2(l)P_1\big(\frac{\log y_1/m}{\log y_1}\big) P_3\big(\frac{\log y_3/l}{\log y_3}\big) }{\sqrt{lm}} \sum_{lab=mc}\frac{1}{a^{1/2+\alpha}b^{1/2+\beta}c^{1/2+\gamma}},
\end{equation*}
$I_{31}^{2}$ is obtained by multiplying $I_{31}^{1}$ with $T^{-(\alpha+\gamma)}$ and changing the shifts $\alpha\leftrightarrow-\gamma$, $\gamma\leftrightarrow-\alpha$, and $I_{31}^{3}$ is obtained by multiplying $I_{31}^{1}$ with $T^{-(\beta+\gamma)}$ and changing the shifts $\beta\leftrightarrow-\gamma$, $\gamma\leftrightarrow-\beta$. 

We shall first work on $ I_{31}^1$. In view of \eqref{Mellin} we get
\begin{eqnarray*}
I_{31}^1&=&\widehat{w}(0)\sum_{i,j}\frac{a_ic_j i!j!}{(\log y_1)^{i}(\log y_3)^j}\Big(\frac{1}{2\pi i}\Big)^2\int_{(1)}\int_{(1)}y_1^{u}y_{3}^v\\
&&\qquad\qquad\sum_{lab=mc}\frac{\mu(m) \mu_2(l)}{m^{1/2+u}l^{1/2+v}a^{1/2+\alpha}b^{1/2+\beta}c^{1/2+\gamma}} \frac{du}{u^{i+1}}\frac{dv}{v^{j+1}}.
\end{eqnarray*}
The Euler product implies that
\begin{eqnarray}\label{55}
&&\sum_{lab=mc}\frac{\mu(m) \mu_2(l)}{m^{1/2+u}l^{1/2+v}a^{1/2+\alpha}b^{1/2+\beta}c^{1/2+\gamma}}\nonumber\\
&&\qquad\qquad=A(\alpha,\beta,\gamma,u,v)\frac{\zeta(1+\alpha+\gamma)\zeta(1+\beta+\gamma)\zeta(1+u+v)^2}{\zeta(1+\alpha+u)\zeta(1+\beta+u)\zeta(1+\gamma+v)^2},
\end{eqnarray}
where $A(\alpha,\beta,\gamma,u,v)$ is an arithmetical factor converging absolutely in a product of half-planes containing the origin. Hence
\begin{equation}\label{40}
I_{31}^1=\widehat{w}(0)\zeta(1+\alpha+\gamma)\zeta(1+\beta+\gamma)\sum_{i,j}\frac{a_ic_j i!j!}{(\log y_1)^{i}(\log y_3)^j}J_{i,j},
\end{equation}
where
\[
J_{i,j}=\Big(\frac{1}{2\pi i}\Big)^2\int_{(1)}\int_{(1)}y_1^{u}y_{3}^v\frac{A(\alpha,\beta,\gamma,u,v)\zeta(1+u+v)^2}{\zeta(1+\alpha+u)\zeta(1+\beta+u)\zeta(1+\gamma+v)^2}\frac{du}{u^{i+1}}\frac{dv}{v^{j+1}}.
\]

Using the Dirichlet series for $\zeta(1+u+v)^2$ and reversing the order of summation and integration, we obtain
\begin{eqnarray}\label{56}
J_{i,j}&=&\sum_{n\leq y_3}\frac{d(n)}{n}\Big(\frac{1}{2\pi i}\Big)^2\int_{(1)}\int_{(1)}\Big(\frac{y_1}{n}\Big)^{u}\Big(\frac{y_3}{n}\Big)^{v}\nonumber\\
&&\qquad\qquad\qquad\frac{A(\alpha,\beta,\gamma,u,v)}{\zeta(1+\alpha+u)\zeta(1+\beta+u)\zeta(1+\gamma+v)^2}\frac{du}{u^{i+1}}\frac{dv}{v^{j+1}}.
\end{eqnarray}
Note that here we are able to restrict the sum over $n$ to $n\leq y_3$ by moving the $v$-integral far to the right. We now move the contours of integration to $\textrm{Re}(u)=\textrm{Re}(v)\asymp \mathcal{L}^{-1}$. Bounding the integrals trivially shows that $J_{i,j}\ll \mathcal{L}^{i+j-2}$. Hence from the Taylor series $A(\alpha,\beta,\gamma,u,v)=A(0,0,0,0,0)+O(\mathcal{L}^{-1})+O(|u|+|v|)$, we can replace $A(\alpha,\beta,\gamma,u,v)$ by $A(0,0,0,0,0)$ in $J_{i,j}$ with an error of size $O(\mathcal{L}^{i+j-3})$. By letting $\alpha=\beta=\gamma=u=v=s$ in \eqref{55}, it is easy to verify that $A(0,0,0,0,0)=1$. The $u$ and $v$ variables in \eqref{56} are now separated so that
\begin{equation}\label{K1}
J_{i,j}=\sum_{n\leq y_3}\frac{d(n)}{n}M_i(\alpha,\beta)M_j(\gamma)+O(\mathcal{L}^{i+j-3}),
\end{equation}
where
\begin{equation*}
M_i(\alpha,\beta)=\frac{1}{2\pi i}\int_{(\mathcal{L}^{-1})}\Big(\frac{y_1}{n}\Big)^{u}\frac{1}{\zeta(1+\alpha+u)\zeta(1+\beta+u)}\frac{du}{u^{i+1}}
\end{equation*}
and
\[
M_j(\gamma)=\frac{1}{2\pi i}\int_{(\mathcal{L}^{-1})}\Big(\frac{y_3}{n}\Big)^{v}\frac{1}{\zeta(1+\gamma+v)^2}\frac{dv}{v^{j+1}}.
\]

The expression $M_i(\alpha,\beta)$ was evaluated in [\textbf{\ref{BCY}}; Lemma 5.7]
\begin{equation}\label{M1}
M_i(\alpha,\beta)=\frac{1}{i!(\log y_1)^2}\frac{d^2}{dx_1dx_2}\bigg[y_{1}^{\alpha x_1+\beta x_2}\Big((x_1+x_2)\log y_1+\log y_1/n\Big)^i\bigg]_{x_1=x_2=0}+O(\mathcal{L}^{i-3}).
\end{equation}
We evaluate $M_j(\gamma)$ with the following lemma.

\begin{lemma}\label{M2}
Suppose $j\geq 2$. Then for some $\nu\asymp (\log\log y_3)^{-1}$ we have
\begin{eqnarray*}
M_j(\gamma)=\frac{1}{j!(\log y_3)^2}\frac{d^2}{dx_3^2}\bigg[y_{3}^{\gamma x_3}\big(x_3\log y_3+\log y_3/n\big)^j\bigg]_{x_3=0}+O(\mathcal{L}^{j-3})+O_\varepsilon\bigg(\Big(\frac{y_3}{n}\Big)^{-\nu}\mathcal{L}^\varepsilon\bigg).
\end{eqnarray*}
\end{lemma}
\begin{proof}
Let $Y = o(T)$ be a large parameter to be chosen later.  By Cauchy's theorem, $M_j(\gamma)$ is equal to the residue at $v=0$ plus integrals over the line segments $\mathcal{C}_1=\{v=\mathcal{L}^{-1}+it,t\in\mathbb{R},|t|\geq Y\}$, $\mathcal{C}_{2}=\{v=\sigma\pm iY,-\frac{c}{\log{Y}}\leq\sigma\leq \mathcal{L}^{-1}\}$ and $\mathcal{C}_3=\{v=-\frac{c}{\log{Y}}+it,|t|\leq Y\}$, where $c$ is some fixed positive constant such that $\zeta(1+\gamma+v)$ has no zeros in the region on the right hand side of the contour determined by the $\mathcal{C}_j$. Furthermore, we require that for such $c$ we have $1/\zeta(\sigma + it) \ll \log(2 + |t|)$ in this region (see [\textbf{\ref{T}}; Theorem 3.11]).  Then the integral over $\mathcal{C}_1$ is 
\begin{equation*}
\ll (\log{Y})^2/Y^{j} \ll_\varepsilon Y^{-2+\varepsilon},
\end{equation*} 
since $j\geq2$. The integral over $\mathcal{C}_2$ is 
\begin{equation*}
\ll (\log{Y})/Y^{j+1} \ll_\varepsilon Y^{-3+\varepsilon}.
\end{equation*}  
Finally, the contribution from $\mathcal{C}_3$ is 
\begin{equation*}
\ll (\log Y)^j \Big(\frac{y_3}{n}\Big)^{-c/\log Y}\ll_\varepsilon \Big(\frac{y_3}{n}\Big)^{-c/\log Y}Y^{\varepsilon}.
\end{equation*}  
Choosing $Y \asymp \mathcal{L}$ gives an error so far of size $O_\varepsilon\big((y_3/n)^{-\nu} \mathcal{L}^{\varepsilon}\big) + O_\varepsilon(\mathcal{L}^{-2+\varepsilon})$.

For the residue at $v=0$, we write this as
\begin{equation*}
\frac{1}{2 \pi i} \oint \Big(\frac{y_3}{n}\Big)^v \frac{1}{\zeta(1+\gamma+v)^2}  \frac{dv}{v^{j+1}},
\end{equation*}
where the contour is a circle of radius $\asymp \mathcal{L}^{-1}$ around the origin. This integral is trivially bounded by $O(\mathcal{L}^{j-2})$, hence by taking the first term in the Taylor series of $\zeta(1+\gamma+v)$ we get
\[
\textrm{Res}_{v=0}=\frac{\gamma^2(\log y_3/n)^j}{j!}+\frac{2\gamma(\log y_3/n)^{j-1}}{(j-1)!}+\frac{(\log y_3/n)^{j-2}}{(j-2)!}+O(\mathcal{L}^{j-3}).
\]
The above main term can be written in a compact form 
\[
\frac{1}{j!(\log y_3)^2}\frac{d^2}{dx_3^2}\bigg[y_{3}^{\gamma x_3}\big(x_3\log y_3+\log y_3/n\big)^j\bigg]_{x_3=0}
\]
and the lemma follows.
\end{proof}

In view of \eqref{K1}, \eqref{M1} and Lemma \ref{M2} we get
\begin{eqnarray*}
J_{i,j}&=&\frac{(\log y_1)^{i-2}(\log y_3)^{j-2}}{i!j!}\frac{d^4}{dx_1dx_2dx_3^2}\bigg[y_{1}^{\alpha x_1+\beta x_2}y_{3}^{\gamma x_3}\\
&&\qquad\qquad\sum_{n\leq y_3}\frac{d(n)}{n}\Big(x_1+x_2+\frac{\log y_1/n}{\log y_1}\Big)^i\Big(x_3+\frac{\log y_3/n}{\log y_3}\Big)^j\bigg]_{\underline{x}=0}\\
&&\qquad\qquad\qquad\qquad+O_\varepsilon\bigg(\mathcal{L}^{i-2+\varepsilon}\sum_{n\leq y_3}\frac{d(n)}{n}\Big(\frac{y_3}{n}\Big)^{-\nu}\bigg)+O(\mathcal{L}^{i+j-3}).
\end{eqnarray*}
Using Lemmas \ref{600} and \ref{601} this is equal to
\begin{eqnarray*}
&&\frac{(\log y_1)^{i-2}(\log y_3)^j}{i!j!}\frac{d^4}{dx_1dx_2dx_3^2}\bigg[y_{1}^{\alpha x_1+\beta x_2}y_{3}^{\gamma x_3}\\
&&\qquad\int_{0}^{1}(1-u)\Big(x_1+x_2+1-\tfrac{\vartheta_3}{\vartheta_1}(1-u)\Big)^i(x_3+u)^jdu\bigg]_{\underline{x}=0}+O_\varepsilon(\mathcal{L}^{i-1+\varepsilon})+O(\mathcal{L}^{i+j-3}).
\end{eqnarray*}
As $j\geq2$, putting this back to \eqref{40} we get
\begin{eqnarray*}
I_{31}^1&=&\frac{\widehat{w}(0)}{(\log y_1)^2}\zeta(1+\alpha+\gamma)\zeta(1+\beta+\gamma)\frac{d^4}{dx_1dx_2dx_3^2}\bigg[y_{1}^{\alpha x_1+\beta x_2}y_{3}^{\gamma x_3}\\
&&\qquad\int_{0}^{1}(1-u)P_1\Big(x_1+x_2+1-\tfrac{\vartheta_3}{\vartheta_1}(1-u)\Big)P_3(x_3+u)du\bigg]_{\underline{x}=0}+O_\varepsilon(T\mathcal{L}^{-1+\varepsilon}).
\end{eqnarray*}

\subsection{Deduction of Lemma \ref{lemmac31}}

Combining $I_{31}^1$, $I_{31}^2$ and $I_{31}^3$ we have
\begin{eqnarray*}
I_{31}&=&\frac{\widehat{w}(0)}{(\log y_1)^2}\frac{d^4}{dx_1dx_2dx_3^2}\bigg[\int_{0}^{1}U_1(\underline{x})(1-u)P_1\Big(x_1+x_2+1-\tfrac{\vartheta_3}{\vartheta_1}(1-u)\Big)P_3(x_3+u)du\bigg]_{\underline{x}=0}\\
&&\qquad\qquad+O_\varepsilon(T\mathcal{L}^{-1+\varepsilon}),
\end{eqnarray*}
where
\begin{eqnarray*}
U_1=\frac{y_{1}^{\alpha x_1+\beta x_2}y_{3}^{\gamma x_3}}{(\alpha+\gamma)(\beta+\gamma)}-\frac{T^{-(\alpha+\gamma)}y_{1}^{-\gamma x_1+\beta x_2}y_{3}^{-\alpha x_3}}{(\alpha+\gamma)(\beta-\alpha)}-\frac{T^{-(\beta+\gamma)}y_{1}^{\alpha x_1-\gamma x_2}y_{3}^{-\beta x_3}}{(\alpha-\beta)(\beta+\gamma)}.
\end{eqnarray*}
Using the identity
\[
\frac{1}{(\alpha+\gamma)(\beta+\gamma)}=\frac{1}{(\alpha+\gamma)(\beta-\alpha)}+\frac{1}{(\alpha-\beta)(\beta+\gamma)}
\]
and the integral formula
\begin{equation}\label{int}
\frac{1-z^{-(\alpha+\beta)}}{\alpha+\beta}=(\log z)\int_{0}^{1}z^{-(\alpha+\beta)t}dt,
\end{equation}
we can write
\begin{eqnarray*}
U_1&=&\frac{y_{1}^{\alpha x_1+\beta x_2}y_{3}^{\gamma x_3}-T^{-(\alpha+\gamma)}y_{1}^{-\gamma x_1+\beta x_2}y_{3}^{-\alpha x_3}}{(\alpha+\gamma)(\beta-\alpha)}-\frac{y_{1}^{\alpha x_1+\beta x_2}y_{3}^{\gamma x_3}-T^{-(\beta+\gamma)}y_{1}^{\alpha x_1-\gamma x_2}y_{3}^{-\beta x_3}}{(\beta-\alpha)(\beta+\gamma)}\\
&=&\frac{y_{1}^{\alpha x_1+\beta x_2}y_{3}^{\gamma x_3}}{\beta-\alpha}\Big(\frac{1-(Ty_{1}^{x_1}y_{3}^{x_3})^{-(\alpha+\gamma)}}{\alpha+\gamma}\Big)-\frac{y_{1}^{\alpha x_1+\beta x_2}y_{3}^{\gamma x_3}}{\beta-\alpha}\Big(\frac{1-(Ty_{1}^{x_2}y_{3}^{x_3})^{-(\beta+\gamma)}}{\beta+\gamma}\Big)\\
&=&\frac{\mathcal{L}}{\beta-\alpha}(1+\vartheta_1 x_1+\vartheta_3 x_3)y_{1}^{\alpha x_1+\beta x_2}y_{3}^{\gamma x_3}\int_{0}^{1}(Ty_{1}^{x_1}y_{3}^{x_3})^{-(\alpha+\gamma)t_1}dt_1\\
&&\qquad\qquad-\frac{\mathcal{L}}{\beta-\alpha}(1+\vartheta_1 x_2+\vartheta_3 x_3)y_{1}^{\alpha x_1+\beta x_2}y_{3}^{\gamma x_3}\int_{0}^{1}(Ty_{1}^{x_2}y_{3}^{x_3})^{-(\beta+\gamma)t_1}dt_1.
\end{eqnarray*}
Hence
\begin{eqnarray*}
I_{31}&=&\frac{\widehat{w}(0)\mathcal{L}}{(\beta-\alpha)(\log y_1)^2}\frac{d^4}{dx_1dx_2dx_3^2}\bigg[\int_{0}^{1}\int_{0}^{1}y_{1}^{\alpha x_1+\beta x_2}y_{3}^{\gamma x_3}(Ty_{1}^{x_1}y_{3}^{x_3})^{-(\alpha+\gamma)t_1}(1+\vartheta_1 x_1+\vartheta_3 x_3)\\
&&\qquad\qquad (1-u)P_1\Big(x_1+x_2+1-\tfrac{\vartheta_3}{\vartheta_1}(1-u)\Big)P_3(x_3+u)dt_1du\bigg]_{\underline{x}=0}\\
&&-\frac{\widehat{w}(0)\mathcal{L}}{(\beta-\alpha)(\log y_1)^2}\frac{d^4}{dx_1dx_2dx_3^2}\bigg[\int_{0}^{1}\int_{0}^{1}y_{1}^{\alpha x_1+\beta x_2}y_{3}^{\gamma x_3}(Ty_{1}^{x_2}y_{3}^{x_3})^{-(\beta+\gamma)t_1}(1+\vartheta_1 x_2+\vartheta_3 x_3)\\
&&\qquad\qquad (1-u)P_1\Big(x_1+x_2+1-\tfrac{\vartheta_3}{\vartheta_1}(1-u)\Big)P_3(x_3+u)dt_1du \bigg]_{\underline{x}=0}+O_\varepsilon(T\mathcal{L}^{-1+\varepsilon}).
\end{eqnarray*}
Changing the roles of the variables $x_1$ and $x_2$ in the second term we obtain
\begin{eqnarray}\label{900}
I_{31}&=&\frac{\widehat{w}(0)\mathcal{L}}{(\log y_1)^2}\frac{d^4}{dx_1dx_2dx_3^2}\bigg[\int_{0}^{1}\int_{0}^{1}V_1(\underline{x})(1+\vartheta_1 x_1+\vartheta_3 x_3)\\
&&\qquad (1-u)P_1\Big(x_1+x_2+1-\tfrac{\vartheta_3}{\vartheta_1}(1-u)\Big)P_3(x_3+u)dt_1du\bigg]_{\underline{x}=0}+O_\varepsilon(T\mathcal{L}^{-1+\varepsilon}),\nonumber
\end{eqnarray}
where
\begin{eqnarray*}
V_1=\frac{y_{1}^{\alpha x_1+\beta x_2}y_{3}^{\gamma x_3}(Ty_{1}^{x_1}y_{3}^{x_3})^{-(\alpha+\gamma)t_1}-y_{1}^{\alpha x_2+\beta x_1}y_{3}^{\gamma x_3}(Ty_{1}^{x_1}y_{3}^{x_3})^{-(\beta+\gamma)t_1}}{\beta-\alpha}.
\end{eqnarray*}
Using \eqref{int} again we have
\begin{eqnarray*}
V_1&=&y_{1}^{\alpha x_1+\beta x_2}y_{3}^{\gamma x_3}(Ty_{1}^{x_1}y_{3}^{x_3})^{-(\alpha+\gamma)t_1}\Big(\frac{1-\big(y_{1}^{-x_1+x_2}(Ty_{1}^{x_1}y_{3}^{x_3})^{t_1}\big)^{-(\beta-\alpha)}}{\beta-\alpha}\Big)\\
&=&\mathcal{L}\Big(-\vartheta_1(x_1-x_2)+t_1(1+\vartheta_1 x_1+\vartheta_3 x_3)\Big)y_{1}^{\alpha x_1+\beta x_2}y_{3}^{\gamma x_3}(Ty_{1}^{x_1}y_{3}^{x_3})^{-(\alpha+\gamma)t_1}\\
&&\qquad\int_{0}^{1}\big(y_{1}^{-x_1+x_2}(Ty_{1}^{x_1}y_{3}^{x_3})^{t_1}\big)^{-(\beta-\alpha)t_2}dt_2.
\end{eqnarray*}
Lemma \ref{lemmac31} follows from this and \eqref{900}.

\section{Proof of Lemma \ref{lemmac23}}\label{section:c23}

\subsection{Reduction to a contour integral}

Recall that $I_{23}(\alpha, \beta,\gamma)$ is defined by \eqref{I23}.  We have
\[
I_{23}(\alpha, \beta,\gamma)=I_{23}^1+I_{23}^2+I_{23}^3,
\]
where
\begin{equation*}
 I_{23}^1= \widehat{w}(0)
\sum_{l,m,n} 
\frac{\mu_2(m)\mu_2(l)P_2(\frac{\log y_2/mn}{\log y_2}) P_3(\frac{\log y_3/l}{\log y_3}) }{\sqrt{lmn}} \sum_{mab=lnc}\frac{1}{a^{1/2+\alpha}b^{1/2+\beta}c^{1/2+\gamma}},
\end{equation*}
$I_{23}^{2}$ is obtained by multiplying $I_{23}^{1}$ with $T^{-(\alpha+\gamma)}$ and changing the shifts $\alpha\leftrightarrow-\gamma$, $\gamma\leftrightarrow-\alpha$, and $I_{23}^{3}$ is obtained by multiplying $I_{23}^{1}$ with $T^{-(\beta+\gamma)}$ and changing the shifts $\beta\leftrightarrow-\gamma$, $\gamma\leftrightarrow-\beta$. 

We first work on $I_{23}^1$. In view of \eqref{Mellin} we get
\begin{eqnarray*}
I_{23}^1&=&\widehat{w}(0)\sum_{i,j}\frac{b_ic_j i!j!}{(\log y_2)^{i}(\log y_3)^j}\Big(\frac{1}{2\pi i}\Big)^2\int_{(1)}\int_{(1)}y_2^{u}y_{3}^v\\
&&\qquad\qquad\sum_{mab=lnc}\frac{\mu_2(m)\mu_2(l)}{(mn)^{1/2+u}l^{1/2+v}a^{1/2+\alpha}b^{1/2+\beta}c^{1/2+\gamma}} \frac{du}{u^{i+1}}\frac{dv}{v^{j+1}}.
\end{eqnarray*}
The arithmetical sum is
\begin{eqnarray}\label{755}
&&\sum_{mab=lnc}\frac{\mu_2(m)\mu_2(l)}{(mn)^{1/2+u}l^{1/2+v}a^{1/2+\alpha}b^{1/2+\beta}c^{1/2+\gamma}}\\
&&\qquad=B(\alpha,\beta,\gamma,u,v)\frac{\zeta(1+\alpha+\gamma)\zeta(1+\beta+\gamma)\zeta(1+\alpha+u)\zeta(1+\beta+u)\zeta(1+u+v)^4}{\zeta(1+\alpha+v)^2\zeta(1+\beta+v)^2\zeta(1+\gamma+u)^2\zeta(1+2u)^2},\nonumber
\end{eqnarray}
where $B(\alpha,\beta,\gamma,u,v)$ is an arithmetical factor converging absolutely in a product of half-planes containing the origin. So
\begin{equation}\label{740}
I_{23}^1=\widehat{w}(0)\zeta(1+\alpha+\gamma)\zeta(1+\beta+\gamma)\sum_{i,j}\frac{b_ic_j i!j!}{(\log y_2)^{i}(\log y_3)^j}K_{i,j},
\end{equation}
where
\[
K_{i,j}=\Big(\frac{1}{2\pi i}\Big)^2\int_{(1)}\int_{(1)}y_2^{u}y_{3}^v\frac{B(\alpha,\beta,\gamma,u,v)\zeta(1+\alpha+u)\zeta(1+\beta+u)\zeta(1+u+v)^4}{\zeta(1+\alpha+v)^2\zeta(1+\beta+v)^2\zeta(1+\gamma+u)^2\zeta(1+2u)^2}\frac{du}{u^{i+1}}\frac{dv}{v^{j+1}}.
\]

Using the Dirichlet series for $\zeta(1+u+v)^4$ and changing the order of summation and integration, we have
\begin{eqnarray}\label{756}
K_{i,j}&=&\sum_{n\leq y_3}\frac{d_4(n)}{n}\Big(\frac{1}{2\pi i}\Big)^2\int_{(1)}\int_{(1)}\Big(\frac{y_2}{n}\Big)^{u}\Big(\frac{y_3}{n}\Big)^{v}\nonumber\\
&&\qquad\qquad\frac{B(\alpha,\beta,\gamma,u,v)\zeta(1+\alpha+u)\zeta(1+\beta+u)}{\zeta(1+\alpha+v)^2\zeta(1+\beta+v)^2\zeta(1+\gamma+u)^2\zeta(1+2u)^2}\frac{du}{u^{i+1}}\frac{dv}{v^{j+1}}.
\end{eqnarray}
Note that here we are able to restrict the sum over $n$ to $n\leq y_3$ by moving the $v$-integral far to the right. We now move the contours of integration to $\textrm{Re}(u)=\textrm{Re}(v)\asymp \mathcal{L}^{-1}$. Bounding the integrals trivially shows that $K_{i,j}\ll \mathcal{L}^{i+j-2}$. Hence from the Taylor series $B(\alpha,\beta,\gamma,u,v)=B(0,0,0,0,0)+O(\mathcal{L}^{-1})+O(|u|+|v|)$, we can replace $B(\alpha,\beta,\gamma,u,v)$ by $B(0,0,0,0,0)$ in $K_{i,j}$ with an error of size $O(\mathcal{L}^{i+j-3})$. By letting $\alpha=\beta=\gamma=u=v=s$ in \eqref{755}, it is easy to verify that $B(0,0,0,0,0)=1$. The $u$ and $v$ variables in \eqref{756} are now separated so that
\begin{equation}\label{J1}
K_{i,j}=\sum_{n\leq y_3}\frac{d_4(n)}{n}N_i(\alpha,\beta,\gamma)N_j(\alpha,\beta)+O(\mathcal{L}^{i+j-3}),
\end{equation}
where
\begin{equation*}
N_i(\alpha,\beta,\gamma)=\frac{1}{2\pi i}\int_{(\mathcal{L}^{-1})}\Big(\frac{y_2}{n}\Big)^{u}\frac{\zeta(1+\alpha+u)\zeta(1+\beta+u)}{\zeta(1+\gamma+u)^2\zeta(1+2u)^2}\frac{du}{u^{i+1}}
\end{equation*}
and
\begin{equation}\label{Nj}
N_j(\alpha,\beta)=\frac{1}{2\pi i}\int_{(\mathcal{L}^{-1})}\Big(\frac{y_3}{n}\Big)^{v}\frac{1}{\zeta(1+\alpha+v)^2\zeta(1+\beta+v)^2}\frac{dv}{v^{j+1}}.
\end{equation}

We evaluate $N_i(\alpha,\beta,\gamma)$ and $N_j(\alpha,\beta)$ with the following lemmas.

\begin{lemma}\label{L1}
Suppose $i\geq 3$. Then we have
\begin{eqnarray*}
N_i(\alpha,\beta,\gamma)&=&\frac{4}{(i-2)!(\log y_2)^2}\frac{d^2}{dx_1^2}\bigg[\int_{0}^{1}\int_{0}^{1} y_{2}^{\gamma x_1-\big(\alpha (1-t_2)+\beta t_2\big)t_1x_1}\Big(\frac{y_2}{n}\Big)^{-\big(\alpha (1-t_2)+\beta  t_2\big)t_1}\\
&&\qquad\qquad\qquad(1-t_1)^{i-2}t_1\Big(x_1\log y_2+(\log y_2/n)\Big)^{i}dt_1dt_2\bigg]_{x_1=0}+O(\mathcal{L}^{i-3}).
\end{eqnarray*}
\end{lemma}
\begin{proof}
An argument on the level of the prime number theorem shows that $N_i(\alpha,\beta,\gamma)$ is equal to the residue at $u=0$ plus an error of size $(\log y_2/n)^{-B}$ for any $B>0$. Since $n\leq y_3$, we have $(\log y_2/n)\geq \log (y_2/y_3) = (\vartheta_2-\vartheta_3)\mathcal{L}$ so that this error is negligible. 

We write the residue at $u=0$ as
\begin{equation*}
\frac{1}{2 \pi i} \oint q^u \frac{\zeta(1+\alpha+u)\zeta(1+\beta+u)}{\zeta(1+\gamma+u)^2\zeta(1+2u)^2}  \frac{du}{u^{i+1}},
\end{equation*}
where $q=y_2/n$ and the contour is a circle of radius $\asymp \mathcal{L}^{-1}$ around the origin. This integral is trivially bounded by $O(\mathcal{L}^{i-2})$. Hence by taking the first terms in the Taylor series of the zeta-functions we have
\begin{eqnarray*}
N_i(\alpha,\beta,\gamma)=\frac{4}{2 \pi i} \oint q^u \frac{(\gamma+u)^2}{(\alpha+u)(\beta+u)} \frac{du}{u^{i-1}}+O(\mathcal{L}^{i-3}).
\end{eqnarray*}

We next use the identity
\begin{equation*}
 (\gamma+u)^2 =  \frac{1}{(\log y_2)^2}\frac{d^2}{dx_{1}^2} y_{2}^{(\gamma+u)x_{1}} \bigg|_{x_{1}=0},
\end{equation*}
to write
\begin{equation*}
N_i(\alpha,\beta,\gamma)=\frac{4}{(\log y_2)^2} \frac{d^2}{dx_{1}^2} y_{2}^{\gamma x_{1}} N(x_{1}) \bigg|_{x_{1}=0},
\end{equation*}
where 
\[
N(x_{1}) = \frac{1}{2 \pi i} \oint \big(y_2^{x_{1}}q\big)^u \frac{1}{(\alpha + u)(\beta+u)} \frac{du}{u^{i-1}}.
\]
Taking the power series gives
\begin{equation*}
N(x_{1}) = \sum_{k \geq 0} \frac{(x_{1}\log y_2 + \log q)^k}{k!} \frac{1}{2 \pi i} \oint \frac{u^{k-i+1}}{(\alpha + u)(\beta+u)} du.
\end{equation*}
As there are three poles inside the contour, it is slightly easier to compute the residue at infinity.  In other words, changing the variable $u \rightarrow 1/u$ yields
\begin{equation*}
N(x_{1}) = \sum_{k\geq 0} \frac{(x_{1}\log y_2 + \log q)^k}{k!} \frac{1}{2 \pi i} \oint \frac{u^{i-k-1}}{(1+\alpha u)(1+\beta u)} du.
\end{equation*}
In view of the power series of $(1+\alpha u)^{-1}$ and $(1+\beta u)^{-1}$, we have
\begin{equation*}
N(x_{1}) = \sum_{k \geq 0} \frac{(x_{1}\log y_2 + \log q)^k}{k!} \sum_{l_1,l_2 \geq 0} (-\alpha)^{l_1}(-\beta)^{l_2} \frac{1}{2 \pi i} \oint u^{i+l_1+l_2-k-1} du.
\end{equation*}
The integral picks out the terms $k = i+l_1+l_2$, giving
\begin{equation}\label{Beta}
N(x_{1}) = (x_{1}\log y_2 + \log q)^{i} \sum_{l_1,l_2 \geq 0} \frac{(-\alpha)^{l_1}(-\beta)^{l_2} (x_{1}\log y_2 + \log q)^{l_1+l_2}}{(i+l_1+l_2)!}.
\end{equation}
We now separate the variables $l_1,l_2$ and $i$ by using the standard beta function and its integral representation
\begin{eqnarray*}
\frac{1}{(i+l_1+l_2)!}& =& B(i-1,l_1+l_2+2) \frac{1}{(i-2)!(l_1+l_2+1)!}\\
&=&B(i-1,l_1+l_2+2)B(l_1+1,l_2+1)\frac{1}{(i-2)!l_1!l_2!}\\
&=&\frac{1}{(i-2)!l_1!l_2!}\int_{0}^{1}\int_{0}^{1}(1-t_{1})^{i-2}t_{1}^{l_1+l_2+1}(1-t_2)^{l_1}t_{2}^{l_2}dt_1dt_2.
\end{eqnarray*}
Putting this into \eqref{Beta} we obtain
\begin{equation*}
N(x_1) =\frac {(x_{1}\log y_2 + \log q)^{i}}{(i-2)!}\int_{0}^{1}\int_{0}^{1}(1-t_{1})^{i-2}t_{1}e^{-\alpha(x_{1}\log y_2+\log q)t_1(1-t_2)-\beta(x_{1}\log y_2+\log q)t_1t_2}dt_1dt_2,
\end{equation*}
and the lemma follows.
\end{proof}

\begin{lemma}\label{L2}
Suppose $j\geq 4$. Then for some $\nu\asymp (\log\log y_3)^{-1}$ we have
\begin{eqnarray*}
N_j(\alpha,\beta)&=&\frac{1}{j!(\log y_3)^4}\frac{d^4}{dx_2^2dx_3^2}\bigg[y_{3}^{\alpha x_2+\beta x_3}\Big((x_2+x_3)\log y_3+\log y_3/n\Big)^j\bigg]_{x_2=x_3=0}\\
&&\qquad\qquad+O(\mathcal{L}^{j-5})+O\bigg(\Big(\frac{y_3}{n}\Big)^{-\nu}\mathcal{L}^\varepsilon\bigg).
\end{eqnarray*}
\end{lemma}
\begin{proof}
Similarly to Lemma \ref{M2}, $N_j(\alpha,\beta)$ equals the residue at $v=0$ plus an error of size $O_\varepsilon\big((y_3/n)^{-\nu} \mathcal{L}^{\varepsilon}\big) + O_\varepsilon(\mathcal{L}^{-2+\varepsilon})$.

For the residue at $v=0$, we write this as
\begin{equation*}
\frac{1}{2 \pi i} \oint \Big(\frac{y_3}{n}\Big)^v \frac{1}{\zeta(1+\alpha+v)^2\zeta(1+\beta+v)^2}  \frac{dv}{v^{j+1}},
\end{equation*}
where the contour is a circle of radius $\asymp \mathcal{L}^{-1}$ around the origin. This integral is trivially bounded by $O(\mathcal{L}^{j-4})$. Hence by taking the first terms in the Taylor series of the zeta-functions we have
\[
\textrm{Res}_{v=0}=\frac{1}{2 \pi i} \oint \Big(\frac{y_3}{n}\Big)^v (\alpha+v)^2(\beta+v)^2 \frac{dv}{v^{j+1}}+O(\mathcal{L}^{j-5}).
\]
The above integral can be written in a compact form as
\[
\frac{1}{j!(\log y_3)^4}\frac{d^4}{dx_2^2dx_3^2}\bigg[y_{3}^{\alpha x_2+\beta x_3}\Big((x_2+x_3)\log y_3+\log y_3/n\Big)^j\bigg]_{x_2=x_3=0},
\]
and the lemma follows.
\end{proof}

In view of \eqref{J1} and Lemmas \ref{L1} and \ref{L2} we get
\begin{eqnarray*}
&&K_{i,j}=\frac{4(\log y_2)^{i-2}(\log y_3)^{j-4}}{(i-2)!j!}\frac{d^6}{dx_1^2dx_2^2dx_3^2}\bigg[\int_{0}^{1}\int_{0}^{1}\\
&&\qquad\Big(\frac{y_2}{y_3}\Big)^{-\big(\alpha (1-t_2)+\beta t_2\big)t_1}y_{2}^{\gamma x_1-\big(\alpha (1-t_2)+\beta t_2\big)t_1x_1}y_{3}^{\alpha x_2+\beta x_3}(1-t_1)^{i-2}t_1\\
&&\qquad\qquad\sum_{n\leq y_3}\frac{d_4(n)}{n}\Big(\frac{y_3}{n}\Big)^{-\big(\alpha (1-t_2)+\beta t_2\big)t_1}\Big(x_1+\frac{\log y_2/n}{\log y_2}\Big)^{i}\Big(x_2+x_3+\frac{\log y_3/n}{\log y_3}\Big)^jdt_1dt_2\bigg]_{\underline{x}=0}\\
&&\qquad\qquad\qquad+O_\varepsilon\bigg(\mathcal{L}^{i-2+\varepsilon}\sum_{n\leq y_3}\frac{d_4(n)}{n}\Big(\frac{y_3}{n}\Big)^{-\nu}\bigg)+O(\mathcal{L}^{i+j-3}).
\end{eqnarray*}
Using Lemmas \ref{600} and \ref{601}, the $O$-terms are $\ll_\varepsilon \mathcal{L}^{i+1+\varepsilon}+\mathcal{L}^{i+j-3}$, while the first term is
\begin{eqnarray*}
&&\frac{2(\log y_2)^{i-2}(\log y_3)^j}{3(i-2)!j!}\frac{d^6}{dx_1^2dx_2^2dx_3^2}\bigg[\int_{0}^{1}\int_{0}^{1}\int_{0}^{1}\\
&&\qquad y_{2}^{\gamma x_1-\big(\alpha (1-t_2)+\beta t_2\big)t_1(1+x_1)}y_{3}^{\alpha x_2+\beta x_3+\big(\alpha (1-t_2)+\beta t_2\big)t_1(1-u)}(1-t_1)^{i-2}t_1(1-u)^3\\
&&\qquad\qquad\Big(x_1+1-\tfrac{\vartheta_3}{\vartheta_2}(1-u)\Big)^{i}(x_2+x_3+u)^jdt_1dt_2du\bigg]_{\underline{x}=0}+O(\mathcal{L}^{i+j-3}).
\end{eqnarray*}
Putting this back to \eqref{740} and using the assumption that $j\geq4$ we get
\begin{eqnarray*}
&&I_{23}^1=\frac{2\widehat{w}(0)}{3(\log y_2)^2}\zeta(1+\alpha+\gamma)\zeta(1+\beta+\gamma)\frac{d^6}{dx_1^2dx_2^2dx_3^2}\bigg[\int_{0}^{1}\int_{0}^{1}\int_{0}^{1}\\
&&\qquad y_{2}^{\gamma x_1-\big(\alpha (1-t_2)+\beta t_2\big)t_1(1+x_1)}y_{3}^{\alpha x_2+\beta x_3+\big(\alpha (1-t_2)+\beta t_2\big)t_1(1-u)}t_1\Big(x_1+1-\tfrac{\vartheta_3}{\vartheta_2}(1-u)\Big)^2(1-u)^3\\
&&\qquad\qquad P_2''\Big(\big(x_1+1-\tfrac{\vartheta_3}{\vartheta_2}(1-u)\big)(1-t_1)\Big)P_3(x_2+x_3+u)dt_1dt_2du\bigg]_{\underline{x}=0}+O_\varepsilon(\mathcal{L}^{-1+\varepsilon}).
\end{eqnarray*}

\subsection{Deduction of Lemma \ref{lemmac23}}

Combining $I_{23}^1$, $I_{23}^2$ and $I_{23}^3$ we have
\begin{eqnarray*}
&&I_{23}=\frac{2\widehat{w}(0)}{3(\log y_2)^2}\frac{d^6}{dx_1^2dx_2^2dx_3^2}\bigg[\int_{0}^{1}\int_{0}^{1}\int_{0}^{1}U_2(\underline{x})t_1\Big(x_1+1-\tfrac{\vartheta_3}{\vartheta_2}(1-u)\Big)^2(1-u)^3\nonumber\\
&&\qquad P_2''\Big(\big(x_1+1-\tfrac{\vartheta_3}{\vartheta_2}(1-u)\big)(1-t_1)\Big)P_3(x_2+x_3+u)dt_1dt_2du\bigg]_{\underline{x}=0}+O_\varepsilon(T\mathcal{L}^{-1+\varepsilon}),
\end{eqnarray*}
where
\begin{eqnarray*}
U_2&=&\frac{y_{2}^{\gamma x_1-\big(\alpha (1-t_2)+\beta t_2\big)t_1(1+x_1)}y_{3}^{\alpha x_2+\beta x_3+\big(\alpha (1-t_2)+\beta t_2\big)t_1(1-u)}}{(\alpha+\gamma)(\beta+\gamma)}\\
&&\qquad-\frac{T^{-(\alpha+\gamma)}y_{2}^{-\alpha x_1-\big(-\gamma (1-t_2)+\beta t_2\big)t_1(1+x_1)}y_{3}^{-\gamma x_2+\beta x_3+\big(-\gamma (1-t_2)+\beta t_2\big)t_1(1-u)}}{(\alpha+\gamma)(\beta-\alpha)}\\
&&\qquad\qquad-\frac{T^{-(\beta+\gamma)}y_{2}^{-\beta x_1-\big(\alpha (1-t_2)-\gamma t_2\big)t_1(1+x_1)}y_{3}^{\alpha x_2-\gamma x_3+\big(\alpha (1-t_2)-\gamma t_2\big)t_1(1-u)}}{(\alpha-\beta)(\beta+\gamma)}.
\end{eqnarray*}

As in the previous section we can write
\begin{eqnarray*}
U_2&=&\frac{y_{2}^{\gamma x_1-\big(\alpha (1-t_2)+\beta t_2\big)t_1(1+x_1)}y_{3}^{\alpha x_2+\beta x_3+\big(\alpha (1-t_2)+\beta t_2\big)t_1(1-u)}}{\beta-\alpha}\\
&&\qquad\qquad\Big(\frac{1-\big(Ty_{2}^{x_1-t_1(1-t_2)(1+x_1)}y_{3}^{x_2+t_1(1-t_2)(1-u)}\big)^{-(\alpha+\gamma)}}{\alpha+\gamma}\Big)\\
&&\qquad-\frac{y_{2}^{\gamma x_1-\big(\alpha (1-t_2)+\beta t_2\big)t_1(1+x_1)}y_{3}^{\alpha x_2+\beta x_3+\big(\alpha (1-t_2)+\beta t_2\big)t_1(1-u)}}{\beta-\alpha}\\
&&\qquad\qquad\qquad\Big(\frac{1-\big(Ty_{2}^{x_1-t_1t_2(1+x_1)}y_{3}^{x_3+t_1t_2(1-u)}\big)^{-(\beta+\gamma)}}{\beta+\gamma}\Big).
\end{eqnarray*}
Applying \eqref{int} leads to
\begin{eqnarray*}
U_2&=&\frac{\mathcal{L}}{\beta-\alpha}\Big(1+\vartheta_2x_1+\vartheta_3x_2-t_1(1-t_2)\big(\vartheta_2(1+x_1)-\vartheta_3(1-u)\big)\Big)\\
&&\qquad y_{2}^{\gamma x_1-\big(\alpha (1-t_2)+\beta t_2\big)t_1(1+x_1)}y_{3}^{\alpha x_2+\beta x_3+\big(\alpha (1-t_2)+\beta t_2\big)t_1(1-u)}\\
&&\qquad\qquad\int_{0}^{1}\big(Ty_{2}^{x_1-t_1(1-t_2)(1+x_1)}y_{3}^{x_2+t_1(1-t_2)(1-u)}\big)^{-(\alpha+\gamma)t_3}dt_3\\
&&\qquad-\frac{\mathcal{L}}{\beta-\alpha}\Big(1+\vartheta_2x_1+\vartheta_3x_3-t_1t_2\big(\vartheta_2(1+x_1)-\vartheta_3(1-u)\big)\Big)\\
&&\qquad\qquad y_{2}^{\gamma x_1-\big(\alpha (1-t_2)+\beta t_2\big)t_1(1+x_1)}y_{3}^{\alpha x_2+\beta x_3+\big(\alpha (1-t_2)+\beta t_2\big)t_1(1-u)}\\
&&\qquad\qquad\qquad \int_{0}^{1}\big(Ty_{2}^{x_1-t_1t_2(1+x_1)}y_{3}^{x_3+t_1t_2(1-u)}\big)^{-(\beta+\gamma)t_3}dt_3.
\end{eqnarray*}
Changing $t_2\rightarrow 1-t_2$ in the first term, and changing the roles of the variables $x_2$ with $x_3$ in the second term yields
\begin{eqnarray*}
I_{23}&=&\frac{2\widehat{w}(0)\mathcal{L}}{3(\log y_2)^2}\frac{d^6}{dx_1^2dx_2^2dx_3^2}\bigg[\mathop{\int}_{[0,1]^4}V_2(\underline{x})\Big(1+\vartheta_2x_1+\vartheta_3x_2-t_1t_2\big(\vartheta_2(1+x_1)-\vartheta_3(1-u)\big)\Big)\\
&&\qquad t_1\Big(x_1+1-\tfrac{\vartheta_3}{\vartheta_2}(1-u)\Big)^2(1-u)^3P_2''\Big(\big(x_1+1-\tfrac{\vartheta_3}{\vartheta_2}(1-u)\big)(1-t_1)\Big)\\
&&\qquad\qquad P_3\big(x_2+x_3+u\big)dt_1dt_2dt_3du\bigg]_{\underline{x}=0}+O_\varepsilon(T\mathcal{L}^{-1+\varepsilon}),
\end{eqnarray*}
where
\begin{eqnarray*}
V_2&=&\frac{1}{\beta-\alpha}\Big(y_{2}^{\gamma x_1-\big(\alpha t_2+\beta  (1-t_2)\big)t_1(1+x_1)}y_{3}^{\alpha x_2+\beta x_3+\big(\alpha t_2+\beta  (1-t_2)\big)t_1(1-u)}\\
&&\qquad\qquad\qquad\qquad\qquad\qquad\big(Ty_{2}^{x_1-t_1t_2(1+x_1)}y_{3}^{x_2+t_1t_2(1-u)}\big)^{-(\alpha+\gamma)t_3}\\
&&\qquad\qquad\qquad -y_{2}^{\gamma x_1-\big(\alpha (1-t_2)+\beta t_2\big)t_1(1+x_1)}y_{3}^{\alpha x_3+\beta x_2+\big(\alpha (1-t_2)+\beta t_2\big)t_1(1-u)}\\
&&\qquad\qquad\qquad\qquad\qquad\qquad\qquad\qquad\qquad \big(Ty_{2}^{x_1-t_1t_2(1+x_1)}y_{3}^{x_2+t_1t_2(1-u)}\big)^{-(\beta+\gamma)t_3}\Big).
\end{eqnarray*}
Using \eqref{int} again implies that $V_2$ equals
\begin{eqnarray*}
&&y_{2}^{\gamma x_1-\big(\alpha t_2+\beta  (1-t_2)\big)t_1(1+x_1)}y_{3}^{\alpha x_2+\beta x_3+\big(\alpha t_2+\beta  (1-t_2)\big)t_1(1-u)}\big(Ty_{2}^{x_1-t_1t_2(1+x_1)}y_{3}^{x_2+t_1t_2(1-u)}\big)^{-(\alpha+\gamma)t_3}\\
&&\qquad \bigg(\frac{1 -\Big(y_{2}^{t_1(2t_2-1)(1+x_1)}y_{3}^{-x_2+ x_3-t_1(2 t_2-1)(1-u)}\big(Ty_{2}^{x_1-t_1t_2(1+x_1)}y_{3}^{x_2+t_1t_2(1-u)}\big)^{t_3}\Big)^{-(\beta-\alpha)}}{\beta-\alpha}\bigg)\\
&=&\mathcal{L}\bigg(-\vartheta_3\big(x_2-x_3)+t_1(2t_2-1)\big(\vartheta_2(1+x_1)-\vartheta_3(1-u)\big)\\
&&\qquad\qquad\qquad\qquad\qquad\qquad+t_3\Big(1+\vartheta_2x_1+\vartheta_3x_2-t_1t_2\big(\vartheta_2(1+x_1)-\vartheta_3(1-u)\big)\Big)\bigg)\\
&&\ y_{2}^{\gamma x_1-\big(\alpha t_2+\beta  (1-t_2)\big)t_1(1+x_1)}y_{3}^{\alpha x_2+\beta x_3+\big(\alpha t_2+\beta  (1-t_2)\big)t_1(1-u)}\big(Ty_{2}^{x_1-t_1t_2(1+x_1)}y_{3}^{x_2+t_1t_2(1-u)}\big)^{-(\alpha+\gamma)t_3}\\
&&\ \ \int_{0}^{1}\Big(y_{2}^{t_1(2t_2-1)(1+x_1)}y_{3}^{-x_2+ x_3-t_1(2 t_2-1)(1-u)}\big(Ty_{2}^{x_1-t_1t_2(1+x_1)}y_{3}^{x_2+t_1t_2(1-u)}\big)^{t_3}\Big)^{-(\beta-\alpha)t_4}dt_4,
\end{eqnarray*}
and the lemma follows.

\section{Proof of Lemma \ref{lemmac3}}\label{section:c3}

\subsection{Reduction to a contour integral}

We first state the twisted fourth moment of the Riemann zeta-function [\textbf{\ref{BBLR}}]. 

\begin{theorem}[Bettin, Bui, Li and Radziwi\l\l]\label{BB}
Suppose $H,K\leq T^{1/4-\varepsilon}$. Then we have
\begin{eqnarray*}
&&\!\!\!\!\!\!\!\!\!\!\sum_{\substack{h\leq H\\k\leq K}}\frac{a_h\overline{a_k}}{\sqrt{hk}}\int_{-\infty}^{\infty}\zeta(\tfrac{1}{2}+\alpha+it)\zeta(\tfrac{1}{2}+\beta+it)\zeta(\tfrac{1}{2}+\gamma-it)\zeta(\tfrac{1}{2}+\delta-it)\Big(\frac{h}{k}\Big)^{it}w(t)dt\\
&&\!\!\!\!\!\!\!\!\!\!\qquad=\sum_{\substack{h\leq H\\k\leq K}}\frac{a_h\overline{a_k}}{\sqrt{hk}}\int_{-\infty}^{\infty}w(t)\bigg\{Z_{\alpha,\beta,\gamma,\delta}(h,k)+\Big(\frac{t}{2\pi}\Big)^{-(\alpha+\gamma)}Z_{-\gamma,\beta,-\alpha,\delta}(h,k)\\
&&\!\!\!\!\!\!\!\!\!\!\qquad\qquad+\Big(\frac{t}{2\pi}\Big)^{-(\alpha+\delta)}Z_{-\delta,\beta,\gamma,-\alpha}(h,k)+\Big(\frac{t}{2\pi}\Big)^{-(\beta+\gamma)}Z_{\alpha,-\gamma,-\beta,\delta}(h,k)\\
&&\!\!\!\!\!\!\!\!\!\!\qquad\qquad+\Big(\frac{t}{2\pi}\Big)^{-(\beta+\delta)}Z_{\alpha,-\delta,\gamma,-\beta}(h,k)+\Big(\frac{t}{2\pi}\Big)^{-(\alpha+\beta+\gamma+\delta)}Z_{-\gamma,-\delta,-\alpha,-\beta}(h,k)\bigg\}dt+O_\varepsilon(T^{1-\varepsilon})
\end{eqnarray*}
uniformly for $\alpha,\beta,\gamma,\delta\ll \mathcal{\mathcal{L}}^{-1}$, where 
\[
Z_{\alpha,\beta,\gamma,\delta}(h,k)=\sum_{kab=hcd}\frac{1}{a^{1/2+\alpha}b^{1/2+\beta}c^{1/2+\gamma}d^{1/2+\delta}}.
\]
\end{theorem}

Recall that $I_{3}(\alpha, \beta,\gamma,\delta)$ is defined by \eqref{I3}. We write
\[
I_{3}(\alpha, \beta,\gamma,\delta)=I_{3}^{1}+I_{3}^2+I_3^3+I_3^4+I_3^5+I_3^6+O_\varepsilon(T^{1-\varepsilon})
\]
correspondingly to the decomposition in Theorem \ref{BB}. We first work on $I_3^1$, which is equal to
\begin{eqnarray*}
\widehat{w}(0)\sum_{l,m}\frac{\mu_2(m)\mu_2(l)P_3\big(\frac{\log y_3/m}{\log y_3}\big) P_3\big(\frac{\log y_3/l}{\log y_3}\big) }{\sqrt{lm}} \sum_{mab=lcd}\frac{1}{a^{1/2+\alpha}b^{1/2+\beta}c^{1/2+\delta}d^{1/2+\gamma}}.
\end{eqnarray*}
In view of \eqref{Mellin} we get
\begin{eqnarray*}
I_3^1&=&\widehat{w}(0)\sum_{i,j}\frac{c_ic_j i!j!}{(\log y_3)^{i+j}}\Big(\frac{1}{2\pi i}\Big)^2\int_{(1)}\int_{(1)}y_3^{u+v}\\
&&\qquad\qquad\sum_{mab=lcd}\frac{\mu_2(m)\mu_2(l)}{m^{1/2+u}l^{1/2+v}a^{1/2+\alpha}b^{1/2+\beta}c^{1/2+\gamma}d^{1/2+\delta}} \frac{du}{u^{i+1}}\frac{dv}{v^{j+1}}.
\end{eqnarray*}
The arithmetical sum is
\begin{eqnarray}\label{I355}
&&\sum_{mab=lcd}\frac{\mu_2(m)\mu_2(l)}{m^{1/2+u}l^{1/2+v}a^{1/2+\alpha}b^{1/2+\beta}c^{1/2+\gamma}d^{1/2+\delta}}\\
&&\qquad=C(\alpha,\beta,\gamma,\delta,u,v)\frac{\zeta(1+\alpha+\gamma)\zeta(1+\alpha+\delta)\zeta(1+\beta+\gamma)\zeta(1+\beta+\delta)\zeta(1+u+v)^4}{\zeta(1+\alpha+v)^2\zeta(1+\beta+v)^2\zeta(1+\gamma+u)^2\zeta(1+\delta+u)^2},\nonumber
\end{eqnarray}
where $C(\alpha,\beta,\gamma,\delta,u,v)$ is an arithmetical factor converging absolutely in a product of half-planes containing the origin. Hence
\begin{equation}\label{I340}
I_3^1=\widehat{w}(0)\zeta(1+\alpha+\gamma)\zeta(1+\alpha+\delta)\zeta(1+\beta+\gamma)\zeta(1+\beta+\delta)\sum_{i,j}\frac{c_ic_j i!j!}{(\log y_3)^{i+j}}L_{i,j},
\end{equation}
where
\begin{eqnarray*}
L_{i,j}&=&\Big(\frac{1}{2\pi i}\Big)^2\int_{(1)}\int_{(1)}y_3^{u+v}\frac{C(\alpha,\beta,\gamma,\delta,u,v)\zeta(1+u+v)^4}{\zeta(1+\alpha+v)^2\zeta(1+\beta+v)^2\zeta(1+\gamma+u)^2\zeta(1+\delta+u)^2}\frac{du}{u^{i+1}}\frac{dv}{v^{j+1}}.
\end{eqnarray*}

Using the Dirichlet series for $\zeta(1+u+v)^4$ and reversing the order of summation and integration, we have
\begin{eqnarray}\label{I356}
L_{i,j}&=&\sum_{n\leq y_3}\frac{d_4(n)}{n}\Big(\frac{1}{2\pi i}\Big)^2\int_{(1)}\int_{(1)}\Big(\frac{y_3}{n}\Big)^{u+v}\nonumber\\
&&\qquad\qquad\frac{C(\alpha,\beta,\gamma,\delta,u,v)}{\zeta(1+\alpha+v)^2\zeta(1+\beta+v)^2\zeta(1+\gamma+u)^2\zeta(1+\delta+u)^2}\frac{du}{u^{i+1}}\frac{dv}{v^{j+1}}.
\end{eqnarray}
Here we are able to restrict the sum over $n$ to $n\leq y_3$ by moving the $u,v$-integrals far to the right. We now move the contours of integration to $\textrm{Re}(u)=\textrm{Re}(v)\asymp \mathcal{L}^{-1}$. Bounding the integrals trivially shows that $L_{i,j}\ll \mathcal{L}^{i+j-4}$. As before we can replace $C(\alpha,\beta,\gamma,\delta,u,v)$ by $C(0,0,0,0,0,0)$ in $L_{i,j}$ with an error of size $O(\mathcal{L}^{i+j-5})$. By letting $\alpha=\beta=\gamma=\delta=u=v=s$ in \eqref{I355}, it is easy to verify that $C(0,0,0,0,0,0)=1$. The $u$ and $v$ variables in \eqref{I356} are now separated so that
\begin{equation*}
L_{i,j}=\sum_{n\leq y_3}\frac{d_4(n)}{n}N_i(\gamma,\delta)N_j(\alpha,\beta)+O(\mathcal{L}^{i+j-5}),
\end{equation*}
where the function $N_j(\alpha,\beta)$ is defined in \eqref{Nj}. Using Lemma \ref{L2} we obtain
\begin{eqnarray*}
L_{i,j}&=&\frac{(\log y_3)^{i+j-8}}{i!j!}\frac{d^8}{dx_1^2dx_2^2dx_3^2dx_4^2}\bigg[y_{3}^{\alpha x_1+\beta x_2+\gamma x_3+\delta x_4}\\
&&\qquad\qquad\sum_{n\leq y_3}\frac{d_4(n)}{n}\Big(x_3+x_4+\frac{\log y_3/n}{\log y_3}\Big)^{i}\Big(x_1+x_2+\frac{\log y_3/n}{\log y_3}\Big)^j\bigg]_{\underline{x}=0}\\
&&\qquad+O_\varepsilon\bigg(\big(\mathcal{L}^{i-4+\varepsilon}+\mathcal{L}^{j-4+\varepsilon}\big)\sum_{n\leq y_3}\frac{d_4(n)}{n}\Big(\frac{y_3}{n}\Big)^{-\nu}\bigg)+O(\mathcal{L}^{i+j-5}).
\end{eqnarray*}
In view of Lemmas \ref{600} and \ref{601}, the $O$-terms are $\ll_\varepsilon \mathcal{L}^{i-1+\varepsilon}+\mathcal{L}^{j-1+\varepsilon}+\mathcal{L}^{i+j-5}$, while the first term is
\begin{eqnarray*}
&&\frac{(\log y_3)^{i+j-4}}{6i!j!}\frac{d^8}{dx_1^2dx_2^2dx_3^2dx_4^2}\bigg[\int_{0}^{1}y_{3}^{\alpha x_1+\beta x_2+\gamma x_3+\delta x_4}(1-u)^3(x_3+x_4+u)^i(x_1+x_2+u)^jdu\bigg]_{\underline{x}=0}\\
&&\qquad\qquad+O(\mathcal{L}^{i+j-5}).
\end{eqnarray*}
Combining this with \eqref{I340}, and using the assumption $i,j\geq4$, we get
\begin{eqnarray*}
I_{3}^1&=&\frac{\widehat{w}(0)}{6(\log y_3)^4}\zeta(1+\alpha+\gamma)\zeta(1+\alpha+\delta)\zeta(1+\beta+\gamma)\zeta(1+\beta+\delta)\frac{d^8}{dx_1^2dx_2^2dx_3^2dx_4^2}\bigg[\int_{0}^{1}\\
&&\qquad y_{3}^{\alpha x_1+\beta x_2+\gamma x_3+\delta x_4}(1-u)^3P_3(x_1+x_2+u)P_3(x_3+x_4+u)du\bigg]_{\underline{x}=0}+O_\varepsilon(\mathcal{L}^{-1+\varepsilon}).
\end{eqnarray*}

\subsection{Deduction of Lemma \ref{lemmac3}}

Recall that $I_3^2$ is essentially obtained by multiplying $I_3^1$ with $T^{-(\alpha+\gamma)}$ and changing the shifts $\alpha\leftrightarrow-\gamma$, $I_{3}^3$ is obtained by multiplying $I_3^1$ with $T^{-(\alpha+\delta)}$ and changing the shifts $\alpha\leftrightarrow-\delta$, $I_3^4$ is obtained by multiplying $I_3^1$ with $T^{-(\beta+\gamma)}$ and changing the shifts $\beta\leftrightarrow-\gamma$, $I_3^5$ is obtained by multiplying $I_3^1$ with $T^{-(\beta+\delta)}$ and changing the shifts $\beta\leftrightarrow-\delta$, and $I_3^6$ is obtained by multiplying $I_3^1$ with $T^{-(\alpha+\beta+\gamma+\delta)}$ and changing the shifts $\alpha\leftrightarrow-\gamma$ and $\beta\leftrightarrow-\delta$. Hence
\begin{eqnarray}\label{U}
I_3(\alpha,\beta,\gamma,\delta)&=&\frac{\widehat{w}(0)}{6(\log y_3)^4}\frac{d^8}{dx_1^2dx_2^2dx_3^2dx_4^2}\bigg[\int_{0}^{1}\\
&&\qquad U(\underline{x})(1-u)^3P_3(x_1+x_2+u)P_3(x_3+x_4+u)du\bigg]_{\underline{x}=0}+O_\varepsilon(\mathcal{L}^{-1+\varepsilon}),\nonumber
\end{eqnarray}
where
\begin{eqnarray*}
U&=&\frac{y_3^{\alpha x_1+\beta x_2+\gamma x_3+\delta x_4}}{(\alpha+\gamma)(\alpha+\delta)(\beta+\gamma)(\beta+\delta)}-\frac{T^{-(\alpha+\gamma)}y_3^{-\gamma x_1+\beta x_2-\alpha x_3+\delta x_4}}{(\alpha+\gamma)(-\gamma+\delta)(\beta-\alpha)(\beta+\delta)}\\
&&\qquad-\frac{T^{-(\alpha+\delta)}y_3^{-\delta x_1+\beta x_2+\gamma x_3-\alpha x_4}}{(-\delta+\gamma)(\alpha+\delta)(\beta+\gamma)(\beta-\alpha)}-\frac{T^{-(\beta+\gamma)}y_3^{\alpha x_1-\gamma x_2-\beta x_3+\delta x_4}}{(\alpha-\beta)(\alpha+\delta)(\beta+\gamma)(-\gamma+\delta)}\\
&&\qquad-\frac{T^{-(\beta+\delta)}y_3^{\alpha x_1-\delta x_2+\gamma x_3-\beta x_4}}{(\alpha+\gamma)(\alpha-\beta)(-\delta+\gamma)(\beta+\delta)}+\frac{T^{-(\alpha+\beta+\gamma+\delta)}y_3^{-\gamma x_1-\delta x_2-\alpha x_3-\beta x_4}}{(\alpha+\gamma)(\gamma+\beta)(\delta+\alpha)(\beta+\delta)}.
\end{eqnarray*}
We write
\begin{eqnarray*}
\frac{y_3^{\alpha x_1+\beta x_2+\gamma x_3+\delta x_4}}{(\alpha+\gamma)(\alpha+\delta)(\beta+\gamma)(\beta+\delta)}&=&\frac{y_3^{\alpha x_1+\beta x_2+\gamma x_3+\delta x_4}}{(\alpha+\gamma)(-\gamma+\delta)(\beta-\alpha)(\beta+\delta)}\\
&&\qquad-\frac{y_3^{\alpha x_1+\beta x_2+\gamma x_3+\delta x_4}}{(-\gamma+\delta)(\alpha+\delta)(\beta+\gamma)(\beta-\alpha)}
\end{eqnarray*}
and
\begin{eqnarray}\label{swap1}
\frac{T^{-(\alpha+\beta+\gamma+\delta)}y_3^{-\gamma x_1-\delta x_2-\alpha x_3-\beta x_4}}{(\alpha+\gamma)(\alpha+\delta)(\beta+\gamma)(\beta+\delta)}&=&\frac{T^{-(\alpha+\beta+\gamma+\delta)}y_3^{-\gamma x_1-\delta x_2-\alpha x_3-\beta x_4}}{(\alpha+\gamma)(-\gamma+\delta)(\beta-\alpha)(\beta+\delta)}\nonumber\\
&&\qquad-\frac{T^{-(\alpha+\beta+\gamma+\delta)}y_3^{-\gamma x_1-\delta x_2-\alpha x_3-\beta x_4}}{(-\gamma+\delta)(\alpha+\delta)(\beta+\gamma)(\beta-\alpha)}.
\end{eqnarray}
Notice that we can change the roles of $x_1$ with $x_2$, or of $x_3$ with $x_4$ in any term of $U(\underline{x})$ without affecting the value of $I_3(\alpha,\beta,\gamma,\delta)$ in \eqref{U}. Applying both changes to the last term in \eqref{swap1}, we can replace $U(\underline{x})$ with
\begin{eqnarray*}
&&\frac{y_3^{\alpha x_1+\beta x_2+\gamma x_3+\delta x_4}}{(-\gamma+\delta)(\beta-\alpha)}\bigg(\frac{1-(Ty_3^{x_1+x_3})^{-(\alpha+\gamma)}}{\alpha+\gamma}\bigg)\bigg(\frac{1-(Ty_3^{x_2+x_4})^{-(\beta+\delta)}}{\beta+\delta}\bigg)\\
&&\qquad-\frac{y_3^{\alpha x_1+\beta x_2+\gamma x_3+\delta x_4}}{(-\gamma+\delta)(\beta-\alpha)}\bigg(\frac{1-(Ty_3^{x_1+x_4})^{-(\alpha+\delta)}}{\alpha+\delta}\bigg)\bigg(\frac{1-(Ty_3^{x_2+x_3})^{-(\beta+\gamma)}}{\beta+\gamma}\bigg).
\end{eqnarray*}
Using \eqref{int} we then get
\begin{eqnarray*}
I_3(\alpha,\beta,\gamma,\delta)&=&\frac{\mathcal{L}^2\widehat{w}(0)}{6(\log y_3)^4}\frac{d^8}{dx_1^2dx_2^2dx_3^2dx_4^2}\bigg[\mathop{\int}_{[0,1]^3}\frac{V_1(\underline{x},t_1,t_2)-V_2(\underline{x},t_1,t_2)}{(-\gamma+\delta)(\beta-\alpha)}\\
&& \qquad(1-u)^3P_3(x_1+x_2+u)P_3(x_3+x_4+u)dt_1dt_2du\bigg]_{\underline{x}=0}+O_\varepsilon(\mathcal{L}^{-1+\varepsilon}),
\end{eqnarray*}
where
\begin{eqnarray*}
V_1=y_3^{\alpha x_1+\beta x_2+\gamma x_3+\delta x_4}(Ty_3^{x_1+x_3})^{-(\alpha+\gamma)t_1}(Ty_3^{x_2+x_4})^{-(\beta+\delta)t_2}\big(1+\vartheta_3(x_1+x_3)\big)\big(1+\vartheta_3(x_2+x_4)\big)
\end{eqnarray*}
and
\begin{eqnarray*}
V_2=y_3^{\alpha x_1+\beta x_2+\gamma x_3+\delta x_4} (Ty_3^{x_1+x_4})^{-(\alpha+\delta)t_1}(Ty_3^{x_2+x_3})^{-(\beta+\gamma)t_2}\big(1+\vartheta_3(x_1+x_4)\big)\big(1+\vartheta_3(x_2+x_3)\big).
\end{eqnarray*}

Again notice that $I_3(\alpha,\beta,\gamma,\delta)$ is unchanged if we swap any of these pairs of variables $x_1\leftrightarrow x_2$, $x_3\leftrightarrow x_4$ and $t_1\leftrightarrow t_2$ in $V_1(\underline{x},t_1,t_2)$ or $V_2(\underline{x},t_1,t_2)$. We next replace $V_1-V_2$ in the integrand with
\begin{eqnarray*}
&&\tfrac{1}{2}\Big(V_1(x_1,x_2,x_3,x_4,t_1,t_2)-V_2(x_2,x_1,x_3,x_4,t_2,t_1)-V_2(x_1,x_2,x_4,x_3,t_1,t_2)\\
&&\qquad\qquad+V_1(x_2,x_1,x_4,x_3,t_2,t_1)\Big),
\end{eqnarray*}
which is
\begin{eqnarray*}
&&\frac{\big(1+\vartheta_3(x_1+x_3)\big)\big(1+\vartheta_3(x_2+x_4)\big)}{2}\bigg(y_3^{\alpha x_1+\beta x_2+\gamma x_3+\delta x_4}(Ty_3^{x_1+x_3})^{-(\alpha+\gamma)t_1}(Ty_3^{x_2+x_4})^{-(\beta+\delta)t_2}\\
&&\qquad-y_3^{\alpha x_2+\beta x_1+\gamma x_3+\delta x_4} (Ty_3^{x_2+x_4})^{-(\alpha+\delta)t_2}(Ty_3^{x_1+x_3})^{-(\beta+\gamma)t_1}\\
&&\qquad\qquad-y_3^{\alpha x_1+\beta x_2+\gamma x_4+\delta x_3} (Ty_3^{x_1+x_3})^{-(\alpha+\delta)t_1}(Ty_3^{x_2+x_4})^{-(\beta+\gamma)t_2}\\
&&\qquad\qquad\qquad+y_3^{\alpha x_2+\beta x_1+\gamma x_4+\delta x_3}(Ty_3^{x_2+x_4})^{-(\alpha+\gamma)t_2}(Ty_3^{x_1+x_3})^{-(\beta+\delta)t_1}\bigg)\\
&&\quad=\frac{\big(1+\vartheta_3(x_1+x_3)\big)\big(1+\vartheta_3(x_2+x_4)\big)}{2}y_3^{\alpha x_1+\beta x_2+\gamma x_3+\delta x_4}\\
&&\qquad\qquad(Ty_3^{x_1+x_3})^{-(\alpha+\gamma)t_1}(Ty_3^{x_2+x_4})^{-(\beta+\delta)t_2}\bigg(1-\Big(T^{t_1-t_2}y_3^{-x_1+x_2+(x_1+x_3)t_1-(x_2+x_4)t_2}\Big)^{-(\beta-\alpha)}\bigg)\\
&&\qquad\qquad\qquad\quad\bigg(1-\Big(T^{t_1-t_2}y_3^{-x_3+x_4+(x_1+x_3)t_1-(x_2+x_4)t_2}\Big)^{-(\delta-\gamma)}\bigg).
\end{eqnarray*}
Using \eqref{int} again and simplifying we obtain Lemma \ref{lemmac3}.

\end{document}